\documentclass[a4paper,oneside,10pt]{article}
\usepackage{amsfonts}
\usepackage{amsmath}
\usepackage{amssymb}
\usepackage{graphicx}
\usepackage{color}
\usepackage{amsthm}
\usepackage{cite}
\usepackage{enumitem}
\usepackage{hyperref}
\usepackage{bbm}
\usepackage[utf8]{inputenc}
\usepackage{array}
\usepackage{empheq}

\newtheoremstyle{uprightstyle}
{3pt} 
{3pt} 
{\upshape} 
{} 
{\bfseries} 
{.} 
{ } 
{} 
\theoremstyle{uprightstyle} 
\newtheorem{theorem}{Theorem}[section]

\newtheorem{definition}[theorem]{Definition}

\newtheorem{lemma}[theorem]{Lemma}

\newtheorem{problem}[theorem]{Problem}
\newtheorem{proposition}[theorem]{Proposition}
\newtheorem{remark}[theorem]{Remark}

\newenvironment{proof}[1][Proof]{\noindent \textbf{#1.} }{\ \ $\Box$}
\numberwithin{equation}{section} 
 
\begin{document}

\title{Optimal control problems of Stochastic Volterra integral equations
under Volatility Ambiguity}
\author{Bingru Zhao \thanks{
Zhongtai Securities Institute for Financial Studies, Shandong University,
Jinan, Shandong 250100, PR China. bingruzhao@mail.sdu.edu.cn.} \and %
Mingshang Hu \thanks{%
Zhongtai Securities Institute for Financial Studies, Shandong University,
Jinan, Shandong 250100, PR China. humingshang@sdu.edu.cn. Research supported
by the National Natural Science Foundation of China (No. 12326603,
11671231). }}

\makeatletter\renewcommand{\@date}{} \makeatother
\maketitle

\textbf{Abstract}. In this
paper, we study the optimal control problems for stochastic Volterra
integral equations driven by $G $-Brownian motion under Volatility Ambiguity. With the help of $G$-stochastic analysis techniques and the
weak convergence methods, we obtain the variation of the cost functional and
the variational inequality. Under the convexity assumptions, we establish
the stochastic maximum principle, which serves as both a necessary and
sufficient condition for optimal control. 



\addcontentsline{toc}{section}{\hspace*{1.8em}Abstract}

\section{Introduction}

\noindent

Let $\left( \Omega ,\mathcal{F},\mathbb{F},\mathbb{P}\right) $ be a complete
probability space and $B$ be a standard Brownian motion, where $\mathbb{F=}$ 
$\mathbb{F}^{B}.$ In many economic and financial models, noise often
exhibits long-range dependence, and state equations possess path-dependent
properties, which make classical Markovian stochastic differential equations
unsuitable. As an important tool for studying these nonlocal structures,
stochastic Volterra integral equations naturally provide a characterization
of these memory effects, which takes the form 
\begin{equation}
	X(t)=\phi (t)+\int_{0}^{t}b(t,s,X(s))ds+\int_{0}^{t}\sigma
	(t,s,X(s))dB_{s},\quad t\in \lbrack 0,T],  \label{introduction1}
\end{equation}%
There are many applications of stochastic Volterra integral equations (SVIEs) in different research areas related to memory effects, which have been extensively studied in recent decades. The
well-posedness of SVIE has been studied in \cite%
{berger1980volterra,ito1979on,pardoux1990stochastic}%
.  Furthermore, in order to study stochastic optimal control problems associated with SVIEs,
Yong \cite{yong2006Wellposedness,yong2008Wellposedness} introduced a
classical backward stochastic Volterra integral equation (BSVIE, for short)
of the following form:%
\begin{equation}
	Y(t)=\psi
	(t)+\int_{t}^{T}g(t,s,Y(s),Z(t,s),Z(s,t))ds-\int_{t}^{T}Z(t,s)dB_{s},
	\label{BSVIE}
\end{equation}%
where $\psi :\left[ 0,T\right] \times \Omega \rightarrow \mathbb{R}$ and $%
g:\Delta \left[ 0,T\right] \times \Omega \times \mathbb{R\times R\rightarrow
	R}$ are measurable mappings.  Moreover, Lin \cite{Lin2002adapted} studied a special
class of Eq.(\ref{BSVIE}) in which $g(\cdot )$ does not depend on $Z(s,t)$
and $\psi (t)\equiv \xi $. The relevant
studies for BSVIE (\ref{BSVIE}) can be founded in \cite%
{Wangzhang2007,HuY2019linear,Ren2010jump,WangH2021risk,Fan_wang2025}
As the adjoint equation of linear SVIEs, BSVIEs provide a powerful tool for studying optimal control problems. Subsequently, the optimal control problems
for SVIEs have attracted significant attention and have been widely applied. For more related works, see \cite%
{WangT2013mean,Hamaguchi_delay,WangT2017forward-backward,Wen_shi_double}) and so on for details.

Motivated by the financial problems with volatility uncertainty, Peng \cite%
{peng2019nonlinear} systematically proposed a time-consistent nonlinear
expectation, called $G$-expectation. Based on these works, $G$-expectation
theory and its applications have been widely developed. Denis et al. \cite%
{Denis2011function} obtained the representation theorem for $G$-expectation $%
\mathbb{\hat{E}}$. The theories of stochastic differential equation by $G$%
-Brownian motion ($G$-SDEs) and backward stochastic differential equation by 
$G$-Brownian ($G$-BSDEs) were established in \cite%
{peng2019nonlinear,Hu2014a}. In addition, by a different method,
Soner et al. \cite{Soner2012Bsde2} obtained a deep result of the existence
and uniqueness theorem for a new type of fully nonlinear BSDE, called 2BSDE.
For other related works, the readers can refer to \cite%
{LiuG2020multi,Jiangl2023stable,Hu2016Qc,Linyq2019reflect,luo2014stochastic,gao2009pathwise,Qutime_vary}%
. Meanwhile, $G$-expectation theory provides a new research framework for
stochastic control problems with volatility uncertainty, and this has
attracted the attention of many researchers. For example, Xu \cite{xu2014G_control}
obtained a necessary condition for optimal control problems of $G$-SDEs and
Biagini et al. \cite{Bagiani2014_Gcontrol} studied strongly robust optimal control problems under
volatility uncertainty. In addition, Hu et al. \cite{Hu2016maximum} obtained
the stochastic maximum principle for stochastic recursive optimal control.
For more developments, the readers can refer to \cite{Buckdan2025meanG,Sun_G_control}.

In this paper, we study the optimal control problem of stochastic Volterra
integral equations under volatility ambiguity, which is driven by $d$%
-dimensional $G$-Brownian motion. More precisely, we consider the following
state equation: for each $t\in[0,T]$, 
\begin{equation}
	\displaystyle X(t)=\mathbb{\phi }(t)+\int_{0}^{t}b(t,s,X(s),u(s))ds+\overset{
		d}{\underset{i,j=1}{\sum }}\int_{0}^{t}h^{ij}(t,s,X(s),u(s))d\left\langle
	B^{i},B^{j}\right\rangle _{s}+\overset{d}{\underset{i=1}{\sum }}
	\int_{0}^{t}\sigma ^{i}(t,s,X(s),u(s))dB_{s}^{i}\text{ }
	\label{introductuion_G_SVIE}
\end{equation}
with the cost functional 
\begin{equation}
	\begin{array}{c}
		\displaystyle J\left( u(\cdot )\right) =\mathbb{\hat{E}}\left[ \varphi
		(X(T))+\int_{0}^{T}g(s,X(s),u(s))ds\right] ,%
	\end{array}%
\end{equation}
where the control domain is convex and $b,h,\sigma,\varphi,g$ are suitable
maps. The objective is to minimize $J(u)$ over the set of admissible control 
$\mathcal{U}[0,T]$.

Observe that the coefficients of $G$-SVIE depend on both $t$ and $s$, which
makes it$\mathrm{\hat{o}}$ formula not applicable. Moreover, there are
additional difficulities in the research of $G$-expectation due to the
uncertainty. On the one hand, the general dominated convergence theorem does
not hold in general, and stochastic processes in $G$-expectation spaces
admit quasi-surely continuous version. This results in difficulties in
deriving convergence estimates as well as in establishing the well-posedness
of Eq.(\ref{introductuion_G_SVIE}). In order to overcome this, by using some 
$G$-stochastic analysis techniques and imposing assumptions (H3), (H5) and
(H6), we obtain the variational equation and the first-order variational
expansion of the cost functional. On the other hand, similar to the
classical probabilistic framework, the adjoint equation of a linear $G$-SDE
is a linear $G$-BSDE. However, note that the $G$-BSDE contains a
non-symmetric $G$-martingale term $K$, which is not differentiable. This
leads to the fact that the necessary conditions derived in the $G$%
-expectation framework depend on the variational equation, which is not
consistent with the stochastic maximum principle. To solve this, by
combining the weak compactness of probability measures (which is the set
that represents $\mathbb{\hat{E}}$) with minimax theorem, we establish the
derivative of the cost functional and the variational inequality under a
reference probability $\tilde{P}$. Furthermore, under each probability
measure $P$, the $G$-Brownian motion reduces to a continuous martingale.
Therefore, it is natural to consider the BSVIE in a general filtration,
which serves as the adjoint equation. Meanwhile, we establish a duality
principle between linear BSVIE in a general filtration and linear $G$-SDE.
Based on these results, we derive the stochastic maximum principle. Finally,
motivated by the derivation of the duality principle, we establish a link
between the Hamiltonian function and the state equation. Using the
representation theorem for $G$-expectation, we show that the derived
stochastic maximum principle is also sufficient under some convex
assumptions.

The paper is organized as follows. In Section 2, we present some
preliminaries related to the $G$-expectation framework. Moreover, some
fundamental results on the $d$-dimensional stochastic Volterra integral
equations driven by $G$-Brownian motion are provided. In Section 3, we
present the formulation of optimal control problems for stochastic Volterra
integral equations under volatility uncertainty. The stochastic maximum
principle is established in Section 4, while sufficient condition for
optimality is derived in Section 5.

\section{Preliminaries}

\noindent

In this section, we recall some basic notions in the framework of $G$%
-expectation. For more developments, the readers can refer to \cite%
{peng2019nonlinear}.

\subsection{\texorpdfstring{$G$-expectation}{G-expectation}}

\noindent

Denote by $\Omega _{T}=C\left( \left[ 0,T\right] ,\mathbb{R}^{d}\right) $
the space of all $\mathbb{R}^{d}$-valued continuous functions on $\left[ 0,T%
\right] $ with $w_{0}=0.$ The canonical process $B$ is defined by $%
B_{t}\left( w\right) =w_{t}$ for each $w\in \Omega _{T}$ and $t\in \left[ 0,T%
\right] $. For each $0\leq t\leq T,$ we set%
\begin{equation*}
Lip\left( \Omega _{t}\right) :=\left\{ \varphi \left(
B_{t_{1}},B_{t_{2}},...,B_{t_{N}}\right) :\varphi \in C_{b.Lip}\left( 
\mathbb{R}^{d\times N}\right) ,t_{1}<\cdot \cdot \cdot <t_{N}\leq t,N\in 
\mathbb{N}\right\} ,
\end{equation*}%
where $C_{b.Lip}\left( \mathbb{R}^{d\times N}\right) $ denotes the space of
all $\mathbb{R}^{d\times N}$-valued bounded Lipschitz functions. Moreover,
we also set 
\begin{equation}
Lip\left( \Omega \right) =\overset{\infty }{\underset{n=1}{\cup }}Lip\left(
\Omega _{n}\right),
\end{equation}

Given a monotonic sublinear function $G:\mathbb{S}_{d}\rightarrow \mathbb{R}$%
, then there exists a bounded, convex and closed subset $\Sigma \subset 
\mathbb{S}_{d}$ such that%
\begin{equation}
G\left( A\right) :=\frac{1}{2}\underset{\gamma \in \Sigma }{\sup }\mathrm{tr}%
\left[ \gamma \gamma ^{T}A\right] ,\text{ }A\in \mathbb{S}_{d},
\label{G_def}
\end{equation}%
where $\mathbb{S}_{d}$ is the set of all $d\times d$ symmetric matrices$.$
It follows that 
\begin{equation}
\underline{\sigma }^{2}I_{d\times d}ds\leq \left[ d\left\langle
B^{i},B^{j}\right\rangle _{s}\right] _{i.j=1}^{d}\leq \bar{\sigma}%
^{2}I_{d\times d}ds.
\end{equation}%
In this paper, we assume that $G$ is non-degenerate, i.e., there is a
constant $\underline{\sigma }^{2}>0$ such that%
\begin{equation}
G\left( A\right) -G\left( B\right) \geq \underline{\sigma }^{2}\mathrm{tr}%
\left[ A-B\right] ,\text{ for each }A\geq B.
\end{equation}

By using a nonlinear parabolic PDE, Peng \cite%
{peng2019nonlinear} construct a consistent $G$-expectation
space $(\Omega _{T},Lip\left( \Omega _{T}\right) $, $\mathbb{\hat{E}}).$ The
canonical process $B$ under $\mathbb{\hat{E}}$ is called the $G$-Brownian
motion. For each $p\geq 1,$ let $L_{G}^{p}\left( \Omega _{T}\right) $ be the
completion of $Lip\left( \Omega _{T}\right) $ under the norm $\left\Vert
\eta \right\Vert _{L_{G}^{p}}=\mathbb{\hat{E}}\left[ \left\vert \eta
\right\vert ^{p}\right] ^{1/p}.$ Furthermore, the $G$-expectation $\mathbb{%
\hat{E}}$ can be extended continuously to $L_{G}^{1}\left( \Omega
_{T}\right) $ under the norm $\left\Vert \cdot \right\Vert _{L_{G}^{1}}$ and 
$(\Omega _{T},Lip\left( \Omega _{T}\right) $, $\mathbb{\hat{E}})$ is called
the $G$-expectation space.

We define%
\begin{equation}
\begin{array}{c}
\mathcal{P=\{}P\text{ }|\text{ }P\text{ is the probability measure on }%
\left( \Omega _{T},\mathcal{B}\left( \Omega _{T}\right) \right) ,\text{ }%
E_{P}[X]\leq \mathbb{\hat{E}}\left[ X\right] ,\text{ for each }X\in
L_{G}^{1}\left( \Omega _{T}\right) .\mathcal{\}},%
\end{array}
\label{hua_P}
\end{equation}%
The $G$-expectation has the following representation theorem.

\begin{theorem}[\negthinspace \negthinspace \protect\cite{Denis2011function}]

\label{G_Def}There exists a unique weakly compact convex set of probability
measures $\mathcal{P}$ (see \ref{hua_P}) such that 
\begin{equation*}
\mathbb{\hat{E}}\left[ X\right] =\underset{P\in \mathcal{P}}{\sup }E_{P}%
\left[ X\right] ,\text{ for each }X\in L_{G}^{1}\left( \Omega _{T}\right) ,
\end{equation*}%
where $\mathcal{P}$ is defined as in (\ref{hua_P}) and $\mathcal{B}\left(
\Omega _{T}\right) =\sigma \left( B_{s}:s\leq T\right) $. $\mathcal{P}$ is
called a set that represents $\mathbb{\hat{E}}$.

\begin{proposition}
\label{weekconverge}[\negthinspace \negthinspace \protect\cite{Denis2011function}] Suppose that the set of probability measures $%
\left\{ P_{n}:n\geq 1\right\} \subset \mathcal{P}$ converges weekly to $P.$
Then for each $X\in L_{G}^{1}\left( \Omega _{T}\right) ,$ we have 
\begin{equation*}
E_{P_{n}}[X]\rightarrow E_{P}[X].
\end{equation*}
\end{proposition}
\end{theorem}

Then we define capacity 
\begin{equation*}
c\left( A\right) =\underset{P\in \mathcal{P}}{\sup }P\left( A\right) ,\text{ 
}A\in \mathcal{B}\left( \Omega_T \right) .
\end{equation*}%
A set $A\in \mathcal{B}\left( \Omega _{T}\right) $ is polar if $c\left(
A\right) =0$. A property holds \textquotedblleft
quasi-surely\textquotedblright\ (q.s.) if it holds except for a polar set.

\begin{definition}
Let $\pi _{T}=\left\{ 0=u_{0}<u_{1}<\cdot \cdot \cdot <u_{N}=T\right\} $ be
the partition of $\left[ 0,T\right] $ for each $N\in \mathbb{N}$. Denote by $%
M_{G}^{0}(0,T;\mathbb{R}^{n})$ the set of processes on $\left[ 0,T\right] $
in the following form:%
\begin{equation*}
\eta _{t}\left( w\right) =\overset{N-1}{\underset{j=0}{\sum }}\xi _{j}\left(
w\right) I_{\left[ u_{j},u_{j+1}\right) }\left( t\right) ,\text{ }\xi
_{j}\in Lip\left( \Omega _{u_{j}}\right) .
\end{equation*}%
For each $p\geq 1,$ denote by $M_{G}^{p}(0,T;\mathbb{R}^{n})$ the completion
of $M_{G}^{0}(0,T;\mathbb{R}^{n})$ with respect to the norm $\left\Vert \eta
\right\Vert _{M_{G}^{p}}=\left( \mathbb{\hat{E}}\left[ \int_{0}^{T}\left%
\vert \eta _{t}\right\vert ^{p}dt\right] \right) ^{1/p}$.
\end{definition}

With the relevant spaces defined above, the stochastic integrals $%
\int_{0}^{\cdot }\eta \left( s\right) dB_{s}$ and $\int_{0}^{\cdot }\xi
\left( s\right) d\left\langle B\right\rangle _{s}$ are well defined for each 
$\eta \in M_{G}^{2}(0,T)$ and $\xi \in M_{G}^{1}(0,T).$

\begin{theorem}
For each $\xi \in H_{G}^{\alpha }(0,T)$ with $\alpha \geq 1$ and $p>0,$ we
can obtain that there exists constants $0<c_{p}<C_{p}<\infty $ such that%
\begin{equation}
\underline{\sigma }^{p}c_{p}\mathbb{\hat{E}}\left[ \left(
\int_{0}^{T}\left\vert \xi \left( s\right) \right\vert ^{2}ds\right) ^{p/2}%
\right] \leq \mathbb{\hat{E}}\left[ \underset{t\in \lbrack 0,T]}{\sup }%
\left\vert \int_{0}^{T}\xi \left( s\right) dB_{s}\right\vert ^{p}\right]
\leq \overline{\sigma }^{2}C_{p}\mathbb{\hat{E}}\left[ \left(
\int_{0}^{T}\left\vert \xi \left( s\right) \right\vert ^{2}ds\right) ^{p/2}%
\right] .  \label{BDG}
\end{equation}
\end{theorem}

Next, we  introduce a $G$-expectation space on the basis of 
$M_{G}^{p}(0,T;\mathbb{R}^{n}).$ Set%
\begin{equation}
\begin{array}{l@{}l}
\tilde{M}_{G}^{p}(0,T;\mathbb{R}^{n})= {}& \left\{ X:\left[ 0,T\right] \times
\Omega \mapsto \mathbb{R}^{n}|X\left( \cdot \right) \in M_{G}^{p}(0,T;%
\mathbb{R}^{n})\right. \\ 
&\left. \text{and }X\left( t\right) \in L_{G}^{p}\left( \Omega
_{t}\right) \text{ for each }t\in \left[ 0,T\right] \right\} .%
\end{array}%
\end{equation}%
The definition of $\tilde{M}_{G}^{p}$-continuity for processes in this space
is given as follows for $p\geq 1$.

\begin{definition}
A process $X(\cdot )\in \tilde{M}_{G}^{p}(0,T)$ is called $\tilde{M}_{G}^{p}$%
-continuous for $p\geq 1$ if it holds that for any $t\in \lbrack 0,T]$, 
\begin{equation}
\lim_{t\rightarrow t_{0}}\hat{\mathbb{E}}[|X(t)-X(t_{0})|^{p}]=0,\text{ for
each fixed }t_{0}\in \lbrack 0,T].  \label{mean_contin}
\end{equation}
For the case $p=2$, $\tilde{M}_{G}^{2}$-continuity is referred to as
mean-square continuity. Then we present a basic property of such processes.
\end{definition}

\begin{lemma}
[\negthinspace \negthinspace \protect\cite{zhao}]Assume that $X(\cdot )\in \tilde{M}_{G}^{2}(0,T)$ is
mean-square continuous. Then we have 
\begin{equation}
\sup_{t\in \lbrack 0,T]}\hat{\mathbb{E}}[|X(t)|^{2}]<\infty .
\end{equation}
\end{lemma}

The following result is the monotone convergence theorem within $G$-expectation %
framework, which is different from the classical case.

\begin{theorem}[\negthinspace \negthinspace \protect\cite{Denis2011function}]

\label{monotone}Let $X_{n},$ $n\geq 1$ and $X$ are $\mathcal{B}\left( \Omega
_{T}\right) $-measureable. Suppose that $\left\{ X_{n}\right\} _{n\geq 1}\in
L_{G}^{1}\left( \Omega _{T}\right) $ satisfy $X_{n}\downarrow X,$ q.s.. Then
we obtain $\mathbb{\hat{E}}\left[ X_{n}\right] \downarrow \mathbb{\hat{E}}%
\left[ X\right] .$
\end{theorem}

\subsection{%
\texorpdfstring{Stochastic
Volterra
Integral
equations
driven
by
$G$-Brownian
motion}{Backward
Stochastic
differential
equations
driven
by G-Brownian
motion}%
}

\noindent

In the following, we review some key results and estimates for
one-dimensional $G$-SVIE established in \cite{zhao}. By analogous arguments, we
further give the corresponding conclusions for the $d$-dimensional setting.
Set 
\begin{equation*}
\Delta \lbrack 0,T]=\left\{ (t,s)\in \lbrack 0,T]^{2}\text{ }|\text{ }0\leq
s\leq t\leq T\right\} .
\end{equation*}%
Let $X\left( \cdot \right) $ be the unique adapted solution to the following 
$G$-SVIE:%
\begin{equation}
X(t)=\phi (t)+\int_{0}^{t}b(t,s,X(s))ds+\underset{i,j=1}{\overset{d}{\sum }}%
\int_{0}^{t}h^{ij}(t,s,X(s))d\left\langle B^{i},B^{j}\right\rangle _{s}+%
\underset{i=1}{\overset{d}{\sum }}\int_{0}^{t}\sigma
^{i}(t,s,X(s))dB_{s}^{i},  \label{G-SVIE_nou}
\end{equation}%
where the coefficients 
\begin{equation*}
\begin{array}{l}
b,h^{ij}:\Delta \times \Omega _{T}\times \mathbb{R}^{n}\rightarrow \mathbb{R}%
^{n}, \\ 
\sigma =\left[ \sigma ^{1},\sigma ^{2},...,\sigma ^{d}\right] :\Delta \times
\Omega _{T}\times \mathbb{R}^{n}\rightarrow \mathbb{R}^{n\times d}, \\ 
\phi :\left[ 0,T\right] \times \Omega _{T}\times \mathbb{R}^{n}\rightarrow 
\mathbb{R}^{n}.%
\end{array}%
\end{equation*}%
The following assumptions are required. For $1\leq i,j\leq d$

(A1) For each $t\in \lbrack 0,T]$ and $x\in \mathbb{R}^{n},\ b(t,\cdot ,x),\
h^{ij}(t,\cdot ,x),\ \sigma ^{i}(t,\cdot ,x)\in M_{G}^{2}(0,t;\mathbb{R}%
^{n}) $;

(A2) There exists a positive constant $L$ such that for each $x,\ y\in 
\mathbb{R}$ and $t,\ s\in \Delta \lbrack 0,T]$, 
\begin{equation*}
|b(t,s,x)-b(t,s,y)|+|h^{ij}(t,s,x)-h^{ij}(t,s,y)|+|\sigma ^{i}(t,s,x)-\sigma
^{i}(t,s,y)|\leq L|x-y|,
\end{equation*}%
\begin{equation*}
|b(t,s,x)|+|h^{ij}(t,s,x)|+|\sigma ^{i}(t,s,x)|\leq L(1+|x|);
\end{equation*}

(A3) For each $(t,s),\ (t^{\prime },s)\in \Delta \lbrack 0,T]$ and $x\in 
\mathbb{R}^{n}$, 
\begin{equation*}
|b(t^{\prime },s,x)-b(t,s,x)|+|h^{ij}(t^{\prime
},s,x)-h^{ij}(t,s,x)|+|\sigma ^{i}(t^{\prime },s,x)-\sigma ^{i}(t,s,x)|\leq
\rho (|t^{\prime }-t|)\left( 1+\left\vert x\right\vert \right) ,
\end{equation*}%
where $\rho :\mathbb{R}^{+}\rightarrow \mathbb{R}^{+}$ is continuous and
strictly increasing with $\rho (0)=0$.

\begin{remark}
Assumption (A3) generalizes condition (H3) in \cite{zhao} by replacing the
original bound with the more general form $\rho (|t^{\prime }-t|)\left(
1+\left\vert x\right\vert \right) .$ By the similar arguments, the well-posedness of $G$-SVIE (\ref{G-SVIE_nou}) still holds under
our weaker assumption (A3).
\end{remark}

\begin{theorem}
\label{non_ujie} Suppose that the assumptions (A1)-(A3) hold and $\phi
(\cdot )\in \tilde{M}_{G}^{2}(0,T;\mathbb{R}^{n}))$ is mean-square
continuous. Then, the $G$-SVIE (\ref{G-SVIE_nou}) admits a unique solution $%
X(\cdot )\in M_{G}^{2}(0,T;\mathbb{R}^{n}))$. Moreover, $X(\cdot )\in \tilde{%
M}_{G}^{2}(0,T;\mathbb{R}^{n}))$ is mean-square continuous.
\end{theorem}

In what follows, we provide a priori estimates that will be used in this
paper.

\begin{proposition}
\label{non_uguji1}Suppose that the assumptions (A1)-(A3) hold and $\phi
(\cdot )\in \tilde{M}_{G}^{2}(0,T;\mathbb{R}^{n}))$ is mean-square
continuous. Let $X(\cdot )$ be the adapted solution of the $G$-SVIE (\ref%
{G-SVIE_nou}). Set $|l\left( t,s,0\right) |=|b\left( t,s,0\right) |+\underset%
{i,j=1}{\overset{d}{\sum }}|h^{ij}\left( t,s,0\right) |.$ Then 
\begin{equation}
\begin{array}{l@{}l}
\displaystyle\hat{\mathbb{E}}[|X(t)|^{2}]\leq {}& \displaystyle C_{1}\left( 
\bar{\sigma},T\right) \Biggl(\hat{\mathbb{E}}[|\phi (t)|^{2}]+\hat{\mathbb{E}%
}\biggl[\biggl(\int_{0}^{T}|l(t,s,0)|ds\biggr)^{2}\biggr]+\underset{i=1}{%
\overset{d}{\sum }}\hat{\mathbb{E}}\biggl[\int_{0}^{T}|\sigma
^{i}(t,s,0)|^{2}ds\biggr]\Biggr) \\[4pt] 
{}& \displaystyle+C_{2}(\bar{\sigma},L,T,d)\int_{0}^{t}\Biggl(\hat{\mathbb{E}}%
[|\phi (r)|^{2}]+\hat{\mathbb{E}}\biggl[\biggl(\int_{0}^{T}|l(r,s,0)|ds%
\biggr)^{2}\biggr]+\underset{i=1}{\overset{d}{\sum }}\hat{\mathbb{E}}\biggl[%
\int_{0}^{T}|\sigma ^{i}(r,s,0)|^{2}ds\biggr]\Biggr)dr,%
\end{array}%
\end{equation}%
where $C_{1}(\bar{\sigma},T)$, $C_{2}(\bar{\sigma},T,L,d)$ are constants
depending on $\bar{\sigma},T$, and $\bar{\sigma},T,L,d$ respectively.
Moreover, if%
\begin{equation*}
\begin{array}{c}
\displaystyle\underset{t\in \lbrack 0,T]}{\sup }\Biggl\{\hat{\mathbb{E}}%
\biggl[\biggl(\int_{0}^{T}|l(t,s,0)|ds\biggr)^{2}\biggr]+\underset{i=1}{%
\overset{d}{\sum }}\hat{\mathbb{E}}\biggl[\int_{0}^{T}|\sigma
^{i}(t,s,0)|^{2}ds\biggr]\Biggr\}<\infty ,%
\end{array}%
\end{equation*}%
then 
\begin{equation}
\begin{array}{c}
\begin{array}{l@{}l}
\displaystyle\hat{\mathbb{E}}[|X(t)|^{2}]\leq {}& \displaystyle C(\bar{\sigma}%
,T,L,d)\underset{t\in \lbrack 0,T]}{\sup }\Biggl\{\hat{\mathbb{E}}[|\phi
(t)|^{2}]+\hat{\mathbb{E}}\biggl[\biggl(\int_{0}^{T}|l(t,s,0)|ds\biggr)^{2}%
\biggr]+\underset{i=1}{\overset{d}{\sum }}\hat{\mathbb{E}}\biggl[%
\int_{0}^{T}|\sigma ^{i}(t,s,0)|^{2}ds\biggr]\Biggr\}.%
\end{array}%
\end{array}%
\end{equation}
\end{proposition}

\begin{proposition}
\label{non_uguji2}Suppose that the assumptions (A1)-(A3) hold and $\phi
_{i}(\cdot )\in \tilde{M}_{G}^{2}(0,T;\mathbb{R}^{n}))$, $i=1,2$ are
mean-square continuous. Let $X_{i}(\cdot )\in \tilde{M}_{G}^{2}(0,T)$ be the
adapted solutions of $G$-SVIEs (\ref{G-SVIE_nou}) corrsponding to the data $%
b_{i},h_{i},\sigma _{i}$ for $i=1,2.$ Then there exists constants $C_{1}(%
\bar{\sigma},T)$ and $C_{2}(\bar{\sigma},T,L)$ such that%
\begin{equation}
\begin{array}{c}
\displaystyle\hat{\mathbb{E}}[|X_{1}(t)-X_{2}(t)|^{2}]\leq \zeta (t)+C_{2}(%
\bar{\sigma},L,T,d)\int_{0}^{t}\zeta (r)dr,\quad t\in \lbrack 0,T],%
\end{array}%
\end{equation}%
where $\hat{\phi}(t)=\phi _{1}(t)-\phi _{2}(t)$, $|\hat{l}%
(t,s)|=|b_{1}(t,s,X_{2}(s))-b_{2}(t,s,X_{2}(s))|+\underset{i,j=1}{\overset{d}%
{\sum }}|h_{1}^{ij}(t,s,X_{2}(s))-h_{2}^{ij}(t,s,X_{2}(s)|$, and%
\begin{equation*}
\begin{array}{c}
\displaystyle\zeta (t)=C_{1}(\bar{\sigma},T)\Biggl(\hat{\mathbb{E}}[|\hat{%
\phi}(t)|^{2}]+\hat{\mathbb{E}}\biggl[\biggl(\int_{0}^{T}|\hat{l}(t,s)|ds%
\biggr)^{2}\biggr]+\underset{i=1}{\overset{d}{\sum }}\hat{\mathbb{E}}\biggl[%
\int_{0}^{T}|\hat{\sigma}^{i}(t,s)|^{2}ds\biggr]\Biggr).%
\end{array}%
\end{equation*}%
Moreover, if%
\begin{equation*}
\begin{array}{c}
\displaystyle\underset{t\in \lbrack 0,T]}{\sup }\Biggl\{\hat{\mathbb{E}}%
\biggl[\biggl(\int_{0}^{T}|\hat{l}(t,s)|ds\biggr)^{2}\biggr]+\underset{i=1}{%
\overset{d}{\sum }}\hat{\mathbb{E}}\biggl[\int_{0}^{T}|\hat{\sigma}%
^{i}(t,s)|^{2}ds\biggr]\Biggr\}<\infty ,%
\end{array}%
\end{equation*}%
we have%
\begin{equation}
\begin{array}{l@{}l}
\hat{\mathbb{E}}[|X_{1}(t)-X_{2}(t)|^{2}]\leq {}& \displaystyle C(\bar{\sigma}%
,T,L,d)\Biggl\{\underset{t\in \lbrack 0,T]}{\sup }\hat{\mathbb{E}}[|\hat{\phi%
}(t)|^{2}]+\underset{t\in \lbrack 0,T]}{\sup }\hat{\mathbb{E}}\biggl[\biggl(%
\int_{0}^{T}|\hat{l}(t,s)|ds\biggr)^{2}\biggr]\Biggr. \\ 
{}& \displaystyle\Biggl.+\underset{t\in \lbrack 0,T]}{\sup }\underset{i=1}{%
\overset{d}{\sum }}\hat{\mathbb{E}}\biggl[\int_{0}^{T}|\hat{\sigma}%
^{i}(t,s)|^{2}ds\biggr]\Biggr\}.%
\end{array}%
\end{equation}
\end{proposition}

\section{Stochastic optimal control problem}

In this section, we introduce the optimal control problem governed by
stochastic Volterra integral equation driven by $G$-Brownian motion ($G$%
-SVIE). Let $\left( \Omega _{T},L_{G}^{1}\left( \Omega _{T}\right) ,\mathbb{%
\hat{E}}\right) $ be the $G$-expectation space$,$ on which $\Omega
_{T}=C\left( \left[ 0,T\right] ;\mathbb{R}^{d}\right) $ and $%
B=(B_{t}^{1},B_{t}^{2},...,B_{t}^{d})_{t\in \lbrack 0.T]}^{T}$ is a $d$%
-dimensional $G$-Brownian motion. Set

\begin{equation}
\begin{array}{c}
\Delta \lbrack 0,T]=\{(t,s)\in \lbrack 0,T]^{2}\text{ }|\text{ }s\leq t\}%
\text{ and }\Delta ^{\ast }[0,T]=\{(t,s)\in \lbrack 0,T]^{2}\text{ }|\text{ }%
s\geq t\}.%
\end{array}%
\end{equation}%
For each interval $[a,b]\subseteq \lbrack 0,T],$ $\Delta (a,b],$ $\Delta
\lbrack a,b)$, $\Delta ^{\ast }[a,b)$ and $\Delta ^{\ast }(a,b]$ also can be
defined similarly.

Next, we first give the definition of admissible controls.

\begin{definition}
Let the control domain $U$ be a nonempty convex of subset of $\mathbb{R}%
^{m}. $ The process $u\left( \cdot \right) :[0,T]\times \Omega
_{T}\rightarrow U$ is called an admissible control if $u\left( \cdot \right)
\in \mathcal{U}[0,T],$ where the set of admissible controls is defined as 
\begin{equation*}
\mathcal{U}[0,T]=M_{G}^{2}(0,T;U)=\left\{ u\in M_{G}^{2}(0,T;\mathbb{R}^{m})%
\text{ }|\text{ }u\left( t\right) \in U\text{ for each }t\in \left[ 0,T%
\right] \right\} \text{.}
\end{equation*}
\end{definition}

Consider the stochastic control system governed by the following $G$-SVIE on 
$[0,T]$: for each $u(\cdot )\in \mathcal{U}[0,T],$ 
\begin{equation}
\displaystyle X(t)=\mathbb{\phi }(t)+\int_{0}^{t}b(t,s,X(s),u(s))ds+\overset{%
d}{\underset{i,j=1}{\sum }}\int_{0}^{t}h^{ij}(t,s,X(s),u(s))d\left\langle
B^{i},B^{j}\right\rangle _{s}+\overset{d}{\underset{i=1}{\sum }}%
\int_{0}^{t}\sigma ^{i}(t,s,X(s),u(s))dB_{s}^{i},\text{ }  \label{SVIE}
\end{equation}%
with the cost functional%
\begin{equation}
\begin{array}{c}
\displaystyle J\left( u(\cdot )\right) =\mathbb{\hat{E}}\left[ \varphi
(X(T))+\int_{0}^{T}g(s,X(s),u(s))ds\right] .%
\end{array}%
\end{equation}%
In the above, the coefficients 
\begin{equation*}
b,h:\Delta \lbrack 0,T]\times \mathbb{R}^{n}\times U\times \Omega
_{T}\rightarrow \mathbb{R}^{n};
\end{equation*}%
\begin{equation*}
\sigma =\left[ \sigma ^{1},\sigma ^{2},...,\sigma ^{d}\right] :\Delta
\lbrack 0,T]\times \mathbb{R}^{n}\times U\times \Omega _{T}\rightarrow 
\mathbb{R}^{n\times d};
\end{equation*}%
\begin{equation*}
\mathbb{\varphi },\phi ,:[0,T]\times \mathbb{R}^{n}\times \Omega
_{T}\rightarrow \mathbb{R}^{n};
\end{equation*}%
\begin{equation*}
g:[0,T]\times \mathbb{R}^{n}\times U\times \Omega _{T}\rightarrow \mathbb{R}%
^{n}.
\end{equation*}%
satisfy the following assumptions. For each $1\leq i,j\leq d.$

(H1) For each $t\in \lbrack 0,T]$ and $\left( x,u\right) \in \mathbb{R}%
^{n}\times U,\ b(t,\cdot ,x,u),\ h^{ij}(t,\cdot ,x,u),\ \sigma ^{i}(t,\cdot
,x,u)\in M_{G}^{2}(0,t;\mathbb{R}^{n}),$ $g\left( \cdot ,x,u\right) \in
M_{G}^{1}(0,T;\mathbb{R}^{n}),$ $\varphi \left( x\right) \in L_{G}^{1}\left(
\Omega _{T}\right) $ and $\phi \left( \cdot \right) \in \tilde{M}%
_{G}^{2}(0,T;\mathbb{R}^{n}))$ is mean-square continuous.

(H2) For each fixed $\left( t,s\right) \in \Delta \lbrack 0,T],$ the map $%
\left( x,u\right) \mapsto (b\left( t,s,x,u\right) ,h^{ij}\left(
t,s,x,u\right) ,\sigma ^{i}\left( t,s,x,u\right) ,g\left( s,x,u\right) ,$ $%
\varphi \left( x\right) )$ are differentiable. The derivatives of $%
b,h^{ij},\sigma ^{i}$ in $\left( x,u\right) $ are uniformly bounded.

(H3) There is a modulus of continuity $\rho :\mathbb{R}^{+}\rightarrow 
\mathbb{R}^{+}$ such that for each $t,t^{\prime },s\in \lbrack 0,T]$ and $%
\left( x,u\right) \in \mathbb{R}^{n}\times U,$%
\begin{equation*}
|l(t^{\prime },s,x,u)-l(t,s,x,u)|\leq \rho (|t^{\prime }-t|)\left(
1+\left\vert x\right\vert +\left\vert u\right\vert \right) ,
\end{equation*}%
where $\rho $ is a continuous and strictly increasing function with $\rho
(0)=0$ and $l=b,b_{x},b_{u},h^{ij},h_{x}^{ij},h_{u}^{ij},\sigma ^{i},\sigma
_{x}^{i},\sigma _{u}^{i}.$

(H4) There exists a constant $C>0$ such that for each $\left( t,s\right) \in
\Delta \lbrack 0,T]$ and $\left( x,u\right) \in \mathbb{R}^{n}\times U,$%
\begin{equation}
\left\vert l(t,s,x,u)\right\vert +\left\vert \mathbb{\varphi }_{x}\left(
x\right) \right\vert \leq C\left( 1+\left\vert x\right\vert +\left\vert
u\right\vert \right) ,
\end{equation}%
where $l=b,h^{ij},\sigma ^{i},g_{x},g_{u}$.

(H5) There exists a modulus of continuity $\varpi :\mathbb{R}^{+}\rightarrow 
\mathbb{R}^{+}$ such that for each $t,s\in \lbrack 0,T]$ and $\left(
x_{1},u_{1}\right) ,\left( x_{2},u_{2}\right) $ $\in \mathbb{R}^{n}\times U,$%
\begin{equation}
\left\vert l\left( t,s,x_{1},u_{1}\right) -l\left( t,s,x_{2},u_{2}\right)
\right\vert \leq \varpi \left( \left\vert x_{1}-x_{2}\right\vert +\left\vert
u_{1}-u_{2}\right\vert \right) ,
\end{equation}%
where $l$ is the derivatives of $b,h^{ij},\sigma ^{i},g,\varphi $ in $\left(
x,u\right) .$

(H6) Let $b_{0}(t,s)=b(t,s,0,0),$ $h_{0}^{ij}(t,s)=h^{ij}(t,s,0,0)$ and $%
\sigma _{0}^{i}(t,s)=\sigma ^{i}(t,s,0,0)$ such that the following holds: 
\begin{equation*}
\begin{array}{c}
\displaystyle\underset{t\in \lbrack 0,T]}{\sup }\hat{\mathbb{E}}\biggl[%
\biggl(\int_{0}^{T}\biggl[|b_{0}(t,s)|+|h_{0}^{ij}(t,s)|\biggr]ds\biggr)^{2}%
\biggr]+\hat{\mathbb{E}}\biggl[\int_{0}^{T}|\sigma _{0}^{i}(t,s)|^{2}ds%
\biggr]<\infty .%
\end{array}%
\end{equation*}

The proposition \ref{3.2} below follows directly from Theorem \ref{non_ujie}
and Proposition \ref{non_uguji1}. In what follows, we denote by $C\left(
a,b\right) $ a positive constant that depends on parameters $a,b,$ and which
may vary from line to line.

\begin{proposition}
\label{3.2}Suppose that assumptions (H1)-(H4) hold. Then for each given $%
u(\cdot )\in \mathcal{U}[0,T],$ the forward stochastic control system (\ref%
{SVIE}) admits a unique solution $X(\cdot )\in \tilde{M}_{G}^{2}(0,T)$.
Moreover, under the assumption (H6), there exists a constant $C$ depending
on $\bar{\sigma},T,L$ such that 
\begin{equation}
\begin{array}{l@{}l}
\displaystyle\underset{t\in \lbrack 0,T]}{\sup }\hat{\mathbb{E}}%
[|X(t)|^{2}]\leq {}& \displaystyle C\underset{t\in \lbrack 0,T]}{\sup }\Biggl\{%
\hat{\mathbb{E}}[|\phi (t)|^{2}]+\hat{\mathbb{E}}\biggl [\biggl(%
\int_{0}^{T}|l_{0}(t,s)|ds\biggr)^{2}\biggr]\Biggr. \\ 
{} & \displaystyle+\overset{d}{\underset{i=1}{\sum }}\hat{\mathbb{E}%
}\biggl[\int_{0}^{T}|\sigma _{0}^{i}(t,s)|^{2}ds\biggr]\Biggr\}+C\hat{%
\mathbb{E}}\biggl[\int_{0}^{T}|u(s)|^{2}ds\biggr],%
\end{array}%
\end{equation}%
where $\left\vert l_{0}\left( t,s\right) \right\vert =\left\vert
b(t,s,0,0)\right\vert +\overset{d}{\underset{i,j=1}{\sum }}%
|h_{0}^{ij}(t,s,0,0)|.$
\end{proposition}

Furthermore, we state our optimal control problem as follows.

\begin{problem}
\label{SG}Under state equation (\ref{SVIE}), find\ a control $\bar{u}(\cdot
)\in \mathcal{U}[0,T]$ such that%
\begin{equation}
J(\bar{u}(\cdot ))=\underset{u(\cdot )\in \mathcal{U}[0,T]}{\inf }J(u(\cdot
)).  \label{J}
\end{equation}
\end{problem}

Any $\bar{u}(\cdot )\in \mathcal{U}[0,T]$ satisfying (\ref{J}) is called an
optimal control of Problem (\ref{SG}). The corresponding state process $\bar{%
X}(\cdot )$ satisfying (\ref{SVIE}) is called an optimal state process and ($%
\bar{X}(\cdot ),$ $\bar{u}(\cdot )$) is called an optimal 2-tuple of Problem
(\ref{SG}).

\section{Stochastic maximum principle and sufficient conditions for
optimality}

In this section, we derive a stochastic maximum principle for Problem \ref%
{SG}.

\subsection{Variational equation and Variational inequality}

Let ($\bar{X}(\cdot ),$ $\bar{u}(\cdot )$) be the optimal 2-tuple of Problem
(\ref{SG}). Take an arbitrary admissible control $u(\cdot )\in \mathcal{U}%
[0,T]$, set $v(\cdot )=u(\cdot )-\bar{u}(\cdot )$ and%
\begin{equation}
u^{\varepsilon }(\cdot )=\bar{u}(\cdot )+\varepsilon v(\cdot ),\text{ }0\leq
\varepsilon \leq 1.
\end{equation}%
Since $\mathcal{U}[0,T]$ is convex, we have $u^{\varepsilon }(\cdot )\in 
\mathcal{U}[0,T].$ Then, we denote by $X^{\varepsilon }(\cdot )$ the state
process of $G$-SVIE (\ref{SVIE}) corresponding to the control process $%
u^{\varepsilon }(\cdot )$. To simplify the statement of this paper, we defne
for each $(t,s)\in \Delta \lbrack 0,T]$:%
\begin{equation}
\begin{array}{cc}
\bar{l}(t,s)=\bar{l}(t,s,\bar{X}(s),\bar{u}(s)), & l^{\varepsilon
}(t,s)=l(t,s,X^{\varepsilon }(s),u^{\varepsilon }(s)); \\ 
\bar{l}_{x}(t,s)=\bar{l}_{x}(t,s,\bar{X}(s),\bar{u}(s)), & 
l_{x}^{\varepsilon }(t,s)=l_{x}(t,s,X^{\varepsilon }(s),u^{\varepsilon }(s));
\\ 
\bar{l}_{u}(t,s)=\bar{l}_{u}(t,s,\bar{X}(s),\bar{u}(s)), & 
l_{x}^{\varepsilon }(t,s)=l_{u}(t,s,X^{\varepsilon }(s),u^{\varepsilon }(s));%
\end{array}
\label{cofficient_def}
\end{equation}%
where $l=b,h^{ij},\sigma ^{i},g$ for each $1\leq i,j\leq d.$ For $l=g,$
since $g$ is independent of $t,$ we adopt the expression $\bar{g}(s)=\bar{g}%
(s,\bar{X}(s),\bar{u}(s))$. The partial derivatives $\bar{g}_{x},\bar{g}%
_{u},g_{x}^{\varepsilon },g_{u}^{\varepsilon }$ are defined analogously. In
this paper, we set 
\begin{equation*}
b_{x}\left( t,s\right) =\left[ 
\begin{array}{ccc}
b_{1x_{1}}\left( t,s\right) , & ..., & b_{1x_{n}}\left( t,s\right) \\ 
\vdots &  & \vdots \\ 
b_{nx_{1}}\left( t,s\right) , & ..., & b_{nx_{n}}\left( t,s\right)%
\end{array}%
\right] .
\end{equation*}%
Similarly, we can define the other partial derivatives.

With these notations established, we consider the following stochastic
Volterra integral equation driven by $G$-Brownian motion:%
\begin{equation}
\begin{array}{l@{}l}
X_{1}(t)= {}& \displaystyle\int_{0}^{t}[\bar{b}_{x}(t,s)X_{1}(s)+\bar{b}%
_{u}(t,s)v(s)]ds+\overset{d}{\underset{i,j=1}{\sum }}\int_{0}^{t}[\bar{h}%
_{x}^{ij}(t,s)X_{1}(s)+\bar{h}_{u}^{ij}(t,s)v(s)]d\left\langle
B^{i},B^{j}\right\rangle _{s} \\ 
{}& \displaystyle+\overset{d}{\underset{i=1}{\sum }}\int_{0}^{t}[\bar{\sigma}%
_{x}^{i}(t,s)X_{1}(s)+\bar{\sigma}_{u}^{i}(t,s)v(s)]dB_{s}^{i}.%
\end{array}
\label{Variation Eq}
\end{equation}%
This is called the variational equation for $G$-SVIE (\ref{SVIE}) and will
be used to derive the necessary conditions for optimality. Under assumptions
(H1)-(H3) and (H6), we derive from Theorem \ref{non_ujie} that Eq.(\ref%
{Variation Eq}) admits a unique solution $X_{1}(\cdot )\in \tilde{M}%
_{G}^{2}(0,T;\mathbb{R}^{n})$ and $\underset{t\in \lbrack 0,T]}{\sup }%
\mathbb{\hat{E}}\left[ \left\vert X_{1}\mathbb{(}t\mathbb{)}\right\vert ^{2}%
\right] <\infty .$ For each $t\in \lbrack 0,T],$ set 
\begin{equation}
\tilde{X}^{\varepsilon }(t)=\frac{X^{\varepsilon }(t)-\bar{X}(t)}{%
\varepsilon }-X_{1}(t).  \label{Xwan}
\end{equation}

\begin{theorem}
\label{Variation_th}Suppose that the assumptions (H1)-(H6) hold. Let $%
X_{1}(\cdot )$ be the unique adapted solution of (\ref{Variation Eq}). Then
we have 
\begin{equation*}
\underset{\varepsilon \rightarrow 0}{\lim }\underset{t\in \lbrack 0,T]}{\sup 
}\mathbb{\hat{E}}\left[ |\tilde{X}\mathbb{^{\varepsilon }(}t\mathbb{)}|^{2}%
\right] =0.
\end{equation*}
\end{theorem}

\begin{proof}
Without loss of generality, we only prove the result for the case $b=0$. The
extension to$\ b\neq 0$ is straightforward and will be omitted. From (\ref%
{SVIE}), (\ref{Variation Eq}) and (\ref{Xwan}), we have%
\begin{equation*}
\begin{array}{l@{}l}
\tilde{X}^{\varepsilon }(t)= {}& \displaystyle\overset{d}{\underset{i,j=1}{%
\sum }}\int_{0}^{t}\left[ \varepsilon ^{-1}\left( h^{ij,\varepsilon }(t,s)-%
\bar{h}^{ij}(t,s)\right) -\bar{h}_{x}^{ij}(t,s)X_{1}(s)-\bar{h}%
_{u}^{ij}(t,s)v(s)\right] d\left\langle B^{i},B^{j}\right\rangle _{s} \\ 
{}& +\displaystyle\overset{d}{\underset{i=1}{\sum }}\int_{0}^{t}\left[
\varepsilon ^{-1}\left( \sigma ^{i,\varepsilon }(t,s)-\bar{\sigma}%
^{i}(t,s)\right) -\bar{\sigma}_{x}^{i}(t,s)X_{1}(s)-\bar{\sigma}%
_{u}^{i}(t,s)v(s)\right] dB_{s}^{i},%
\end{array}%
\end{equation*}%
where the generator is defined in (\ref{cofficient_def}). Note that%
\begin{equation*}
\begin{array}{ll}
\varepsilon ^{-1}\left( l^{\varepsilon }(t,s)-\bar{l}(t,s)\right) = %
\displaystyle\int_{0}^{1}l_{x}^{\theta }(t,s)d\theta \frac{X^{\varepsilon
}(s)-\bar{X}(s)}{\varepsilon }+\displaystyle\int_{0}^{1}l_{u}^{\theta
}(t,s)d\theta v(s),%
\end{array}%
\end{equation*}%
where $l_{x}^{\theta }(t,s)=l_{x}(t,s,\bar{X}(s)+\theta \left(
X^{\varepsilon }(s)-\bar{X}(s)\right) ,\bar{u}(s)+\theta \varepsilon v(s))$, 
$l_{u}^{\theta }(t,s)=l_{u}(t,s,\bar{X}(s)+\theta (X^{\varepsilon }(s)-\bar{X%
}(s)),\bar{u}(s)+\theta \varepsilon v(s)),$ and $l=\sigma ^{i},h^{ij}.$
Since $\frac{X^{\varepsilon }(s)-\bar{X}(s)}{\varepsilon }=\tilde{X}%
^{\varepsilon }(s)+X_{1}\left( s\right) ,$ we obtain%
\begin{equation}
\begin{array}{l@{}l}
\tilde{X}^{\varepsilon }(t)= {}& \displaystyle\overset{d}{\underset{i,j=1}{%
\sum }}\int_{0}^{t}\biggl[\int_{0}^{1}h_{x}^{ij,\theta }(t,s)d\theta \tilde{X%
}^{\varepsilon }(s)+R_{h^{ij}}^{\theta }(t,s))\biggr]d\left\langle
B^{i},B^{j}\right\rangle _{s} \\ 
{}& +\displaystyle\overset{d}{\underset{i=1}{\sum }}\int_{0}^{t}\left[
\int_{0}^{1}\sigma _{x}^{i,\theta }(t,s)d\theta \tilde{X}^{\varepsilon
}(s)+R_{\sigma ^{i}}^{\theta }(t,s)\right] dB_{s}^{i},%
\end{array}
\label{Var_1}
\end{equation}%
where%
\begin{equation*}
\begin{array}{c}
\displaystyle R_{h^{ij}}^{\theta }(t,s)=\int_{0}^{1}\left( h_{x}^{ij,\theta
}(t,s)-\bar{h}_{x}^{ij}(t,s)\right) d\theta X_{1}(s)+\int_{0}^{1}\left(
h_{u}^{ij,\theta }(t,s)-\bar{h}_{u}^{ij}(t,s)\right) d\theta v(s)%
\end{array}%
\end{equation*}%
and%
\begin{equation*}
\begin{array}{c}
\displaystyle R_{\sigma ^{i}}^{\theta }(t,s)=\int_{0}^{1}(\sigma
_{x}^{i,\theta }(t,s)-\bar{\sigma}_{x}^{i}(t,s))d\theta
X_{1}(s)+\int_{0}^{1}\left( \sigma _{u}^{i,\theta }(t,s)-\bar{\sigma}%
_{u}^{i}(t,s)\right) d\theta v(s).%
\end{array}%
\end{equation*}%
By Proposition \ref{non_uguji1} and (H2), we derive%
\begin{equation}
\begin{array}{l@{}l}
\underset{t\in \lbrack 0,T]}{\sup }\mathbb{\hat{E}}\bigl[|\tilde{X}\mathbb{%
^{\varepsilon }(}t\mathbb{)}|^{2}\bigr] &{} \leq \displaystyle C\underset{t\in
\lbrack 0,T]}{\sup }\Biggl\{\hat{\mathbb{E}}\biggl[\biggl(\int_{0}^{T}%
\overset{d}{\underset{i,j=1}{\sum }}|R_{h^{ij}}^{\theta }(t,s)|ds\biggr)^{2}%
\biggr]+\overset{d}{\underset{i=1}{\sum }}\hat{\mathbb{E}}\biggl[%
\int_{0}^{T}|R_{\sigma ^{i}}^{\theta }(t,s)|^{2}ds\biggr]\Biggr\} \\ 
& {}\leq \displaystyle C\Biggl\{\underset{t\in \lbrack 0,T]}{\sup }\hat{%
\mathbb{E}}[|X_{1}(t)|^{2}]+\hat{\mathbb{E}}\biggl[\int_{0}^{T}v^{2}(s)ds%
\biggr]\Biggr\}<\infty ,%
\end{array}
\label{uni_bound}
\end{equation}%
where $C$ depends on $\bar{\sigma},L,T,d.$

Then, it suffices to prove $\underset{\varepsilon \rightarrow 0}{\lim }%
\underset{t\in \lbrack 0,T]}{\sup }\biggl\{\hat{\mathbb{E}}\Bigl[\left(
\int_{0}^{T}|R_{h^{ij}}^{\theta }(t,s)|ds\right) ^{2}\Bigr]+\hat{\mathbb{E}}%
\left[ \int_{0}^{T}|R_{\sigma ^{i}}^{\theta }(t,s)|^{2}ds\right] \biggr\}=0.$
For each $N\in \mathbb{N},$ define%
\begin{equation*}
\begin{array}{c}
A_{N}(s)=\left\{ |\tilde{X}\mathbb{^{\varepsilon }}(s)|+|X_{1}(s)|+|v(s)|%
\leq N\right\} .%
\end{array}%
\end{equation*}%
Since the derivatives of $\sigma ^{i}$ with respect to $(x,u)$ is bounded,
it is easy to verify that 
\begin{equation*}
\begin{array}{c}
|R_{\sigma ^{i}}^{\theta }(t,s)|\leq C\left\vert X_{1}(s)+v(s)\right\vert .%
\end{array}%
\end{equation*}%
Recall that $X^{\varepsilon }(t)-\bar{X}(t)=\varepsilon \bigl(\tilde{X}%
^{\varepsilon }(t)+X_{1}(t)\bigr)$. Combining this with assumption (H5), we
further obtain 
\begin{equation*}
\begin{array}{c}
|R_{\sigma ^{i}}^{\theta }(t,s)|\leq \displaystyle\int_{0}^{1}\bar{w}\Bigl(%
\theta \varepsilon \bigl\vert\tilde{X}^{\varepsilon }(s)+X_{1}(s)\bigr\vert%
+\theta \varepsilon \left\vert v(s)\right\vert \Bigr)d\theta \left\vert
X_{1}(s)+v(s)\right\vert \\ 
\text{ \ \ }\leq \bar{w}\Bigl(\varepsilon \bigl\vert\tilde{X}^{\varepsilon
}(s)+X_{1}(s)\bigr\vert+\varepsilon |v(s)|\Bigr)\bigl\vert X_{1}(s)+v(s)%
\bigr\vert.%
\end{array}%
\end{equation*}%
Accordingly, by (H2) we have%
\begin{equation}
\begin{array}{l}
\displaystyle\underset{t\in \lbrack 0,T]}{\sup }\hat{\mathbb{E}}\biggl[%
\int_{0}^{T}|R_{\sigma ^{i}}^{\theta }(t,s)|^{2}ds\biggr] \\ 
=\displaystyle\underset{t\in \lbrack 0,T]}{\sup }\hat{\mathbb{E}}\biggl[%
\int_{0}^{T}|R_{\sigma ^{i}}^{\theta
}(t,s)|^{2}(I_{A_{N}(s)}+I_{A_{N}(s)}^{c})ds\biggr] \\ 
\leq \displaystyle\hat{\mathbb{E}}\biggl[\int_{0}^{T}\Bigl\vert\bar{w}\Bigl(%
\varepsilon \bigl\vert\tilde{X}^{\varepsilon }(s)+X_{1}(s)\bigr\vert%
+\varepsilon \bigr\vert v(s)\bigr\vert\Bigr)^{2}\Bigr\vert\left\vert
X_{1}(s)+v(s)\right\vert ^{2}I_{A_{N}(s)}ds\biggr] \\ 
\text{ \ }\displaystyle+C\hat{\mathbb{E}}\biggl[\int_{0}^{T}\left\vert
X_{1}(s)+v(s)\right\vert ^{2}I_{A_{N}(s)}^{c}ds\biggr] \\ 
\leq \displaystyle\left\vert \bar{w}\left( N\varepsilon \right) \right\vert
^{2}\hat{\mathbb{E}}\biggl[\int_{0}^{T}\left\vert X_{1}(s)+v(s)\right\vert
^{2}ds\biggr]+C\hat{\mathbb{E}}\biggl[\int_{0}^{T}\left\vert
X_{1}(s)+v(s)\right\vert ^{2}I_{A_{N}(s)}^{c}ds\biggr].%
\end{array}
\label{Var_2}
\end{equation}%
Since $X_{1}(\cdot )+v(\cdot )\in M_{G}^{2}(0,T),$ we obtain from
Theorem 4.7 in \cite{Hu2016Qc} that there exists $M_{0}\in \mathbb{N}$ such that for
each $M\geq M_{0}$ and $\delta >0,$%
\begin{equation}
\begin{array}{c}
\displaystyle\hat{\mathbb{E}}\biggl[\int_{0}^{T}\left\vert
X_{1}(s)+v(s)\right\vert ^{2}I_{\left\{ \left\vert X_{1}(s)+v(s)\right\vert
>M_{0}\right\} }ds\biggr]<\frac{\delta }{4}.%
\end{array}
\label{Var_3}
\end{equation}%
Furthermore, by (\ref{uni_bound}), it is easy to check that there exists $%
N_{0}\in \mathbb{N}$ such that for each $N\geq N_{0},$%
\begin{equation}
\begin{array}{c}
\displaystyle M_{0}^{2}\hat{\mathbb{E}}\biggl[\int_{0}^{T}I_{A_{N}(s)}^{c}ds%
\biggr]\leq M_{0}^{2}\frac{3}{N^{2}}\hat{\mathbb{E}}\biggl[\int_{0}^{T}|%
\tilde{X}\mathbb{^{\varepsilon }}(s)|^{2}+|X_{1}(s)|^{2}+|v(s)|^{2}ds\biggr]%
\leq \frac{3M_{0}^{2}C}{N^{2}}<\frac{\delta }{4}.%
\end{array}
\label{Var_4}
\end{equation}%
From (\ref{Var_3}) and (\ref{Var_4}), we deduce%
\begin{equation}
\begin{array}{l}
\displaystyle\hat{\mathbb{E}}\biggl[\int_{0}^{T}\left\vert
X_{1}(s)+v(s)\right\vert ^{2}I_{A_{N_{0}}(s)}^{c}ds\biggr] \\ 
\displaystyle\leq \hat{\mathbb{E}}\biggl[\int_{0}^{T}\left\vert
X_{1}(s)+v(s)\right\vert ^{2}I_{\left\{ \left\vert X_{1}(s)+v(s)\right\vert
>M_{0}\right\} }I_{A_{N_{0}}(s)}^{c}ds\biggr]+\hat{\mathbb{E}}\biggl[%
\int_{0}^{T}M_{0}^{2}I_{A_{N_{0}}(s)}^{c}ds\biggr]<\frac{\delta }{2}.%
\end{array}
\label{Var_5}
\end{equation}%
Set $N=N_{0}$ in (\ref{Var_2}). Then by (H5), it holds that there exists a
constant $\tilde{\delta}>0$ such that for each $\tilde{\delta}>\varepsilon
>0,$ 
\begin{equation}
\begin{array}{c}
\displaystyle\left\vert \bar{w}\left( N_{0}\varepsilon \right) \right\vert
^{2}\hat{\mathbb{E}}\biggl[\int_{0}^{T}\left\vert X_{1}(s)+v(s)\right\vert
^{2}ds\biggr]<\frac{\delta }{2}.%
\end{array}
\label{Var_6}
\end{equation}%
Since $\delta $ is arbitrary, we conclude from (\ref{Var_2}),(\ref{Var_5})
and (\ref{Var_6}) that $\underset{\varepsilon \rightarrow 0}{\lim }\underset{%
t\in \lbrack 0,T]}{\sup }\hat{\mathbb{E}}\left[ \int_{0}^{T}|R_{\sigma
^{i}}^{\theta }(t,s)|^{2}ds\right] =0.$ By a similar argument, we also have $%
\underset{\varepsilon \rightarrow 0}{\lim }\underset{t\in \lbrack 0,T]}{\sup 
}\hat{\mathbb{E}}\Bigl[\left( \int_{0}^{T}|R_{h^{ij}}^{\theta
}(t,s)|ds\right) ^{2}\Bigr]=0$. The proof is complete.
\end{proof}

Denote 
\begin{equation}
\begin{array}{c}
\displaystyle\Gamma \left( X,u\right) =\mathbb{\varphi }(X(T))+%
\int_{0}^{T}g(s,X(s),u(s))ds%
\end{array}
\label{Def_J}
\end{equation}%
and 
\begin{equation}
\begin{array}{c}
\displaystyle\Psi \left( u\right) =\mathbb{\varphi }_{x}(\bar{X}%
(T))X_{1}(T)+\int_{0}^{T}\left[ \bar{g}_{x}(s)X_{1}(s)+\bar{g}_{u}(s)\left(
u(s)-\bar{u}(s)\right) \right] ds.%
\end{array}%
\end{equation}%
For each $u(\cdot )\in \mathcal{U}[0,T],$ set%
\begin{equation}
\mathcal{\tilde{P}}\left( X,u\right) \mathcal{=}\left\{ P\in \mathcal{P}%
\text{ }|\text{ }E_{P}\mathbb{[}\Gamma \left( X,u\right) ]\mathbb{=\hat{E}[}%
\Gamma \left( X,u\right) \mathbb{]}\right\} .  \label{set_of_p}
\end{equation}%
We first present the following lemmas, which play an important role in the
rest of this section.

\begin{lemma}
\label{weakness}For each $u(\cdot )\in \mathcal{U}[0,T]$ and $X(\cdot )\in 
\tilde{M}_{G}^{2}(0,T;\mathbb{R}^{n}),$ the family of probability measures $%
\mathcal{\tilde{P}}\left( X,u\right) $ is convex and weakly compact.

\begin{proof}
Since the expectation $E_{P}$ is linear in the probability measure $P$, the
convexity of $\mathcal{\tilde{P}}\left( X,u\right) $ follows directly.
Recall that the family of probability measures $\mathcal{P}$ is weakly
compact, it remains to show that $\mathcal{\tilde{P}}\left( X,u\right) $ is
a weakly closed subset of $\mathcal{P}$. Assume that the set of probability
measures $\left\{ P_{n}\right\} _{n\geq 1}\subset \mathcal{\tilde{P}}\left(
X,u\right) \subset \mathcal{P}$ converge weekly to $P.$ Then by Proposition %
\ref{weekconverge}, we have $P\in \mathcal{P}$ and 
\begin{equation}
\begin{array}{c}
E_{P_{n}}[\xi ]\rightarrow E_{P}[\xi ],\text{ for each }\xi \in
L_{G}^{1}\left( \Omega _{T}\right).%
\end{array}
\label{week_1}
\end{equation}
For each $n\geq 1$, since $E_{P_{n}}[\Gamma \left( X,u\right) ]=\mathbb{\hat{%
E}[}\Gamma \left( X,u\right) \mathbb{]}$ and $\Gamma \left( X,u\right) \in
L_{G}^{2}(\Omega _{T}),$ we derive from (\ref{week_1}) that $E_{P}[\Gamma
\left( X,u\right) ]=\mathbb{\hat{E}[}\Gamma \left( X,u\right) \mathbb{]},$
which implies that $P\in \mathcal{\tilde{P}}\left( X,u\right) .$ The proof
is complete.
\end{proof}
\end{lemma}

The next result presents the first-order variational
expansion of $\Gamma \left( X,u\right) $
under the convex control perturbations.

\begin{lemma}
\label{Taylor}Assume that (H1)-(H6) hold. Then we have 
\begin{equation}
\begin{array}{c}
\Gamma \left( X^{\varepsilon },u^{\varepsilon }\right) =\Gamma \left( \bar{X}%
,\bar{u}\right) +\varepsilon \Psi \left( u\right) +\mathcal{R}(\varepsilon )%
\end{array}%
\end{equation}%
and $\hat{\mathbb{E}}\left[ \left\vert \mathcal{R}(\varepsilon )\right\vert %
\right] =o\left( \varepsilon \right) $, where%
\begin{equation*}
\begin{array}{l@{}l}
\displaystyle\mathcal{R}(\varepsilon )= {}& \displaystyle\varepsilon
\int_{0}^{1}\left[ \mathbb{\varphi }_{x}\left( \theta ,\varepsilon \right) -%
\mathbb{\varphi }_{x}(\bar{X}(T))\right] d\theta X_{1}(T)+\varepsilon
\int_{0}^{T}\int_{0}^{1}\left[ g_{x}\left( \theta ,s\right) -g_{x}(s)\right]
d\theta X_{1}(s)ds \\ 
{}& \displaystyle+\varepsilon \int_{0}^{T}\int_{0}^{1}\left[ g_{u}\left(
\theta ,s\right) -g_{u}(s)\right] d\theta v(s)ds+\varepsilon \int_{0}^{1}%
\mathbb{\varphi }_{x}\left( \theta ,\varepsilon \right) d\theta \tilde{X}%
\mathbb{^{\varepsilon }}\left( T\right) +\varepsilon
\int_{0}^{T}\int_{0}^{1}g_{x}\left( \theta ,s\right) d\theta \tilde{X}%
\mathbb{^{\varepsilon }}\left( s\right) ds,%
\end{array}%
\end{equation*}%
\begin{equation*}
\begin{array}{c}
l\left( \theta ,t\right) =l\Bigl(t,\bar{X}(t)+\theta \varepsilon \bigl(%
X_{1}(t)+\tilde{X}\mathbb{^{\varepsilon }}\left( t\right) \bigr),\bar{u}%
(t)+\theta \varepsilon v\left( t\right) \Bigr),\text{ }l=g_{x},g_{u},%
\end{array}%
\end{equation*}%
\begin{equation*}
\begin{array}{c}
\mathbb{\varphi }_{x}\left( \theta ,\varepsilon \right) =\varphi _{x}\Bigl(%
\bar{X}(T)+\theta \varepsilon \bigl(X_{1}(T)+\tilde{X}\mathbb{^{\varepsilon }%
}\left( T\right) \bigr)\Bigr).%
\end{array}%
\end{equation*}
\end{lemma}

\begin{proof}
Note that 
\begin{equation}
\begin{array}{c}
\displaystyle\Gamma \left( X^{\varepsilon },u^{\varepsilon }\right) -\Gamma
\left( \bar{X},\bar{u}\right) =\mathbb{\varphi }(X^{\varepsilon }(T))-%
\mathbb{\varphi }(\bar{X}(T))+\int_{0}^{T}\left[ g^{\varepsilon }(s)-\bar{g}%
(s)\right] ds,%
\end{array}%
\end{equation}%
where $g^{\varepsilon }(s)$ and $\bar{g}(s)$ are defined in (\ref%
{cofficient_def}). By simple calculation, we first analyze the term
involving $\varphi $:%
\begin{equation}
\begin{array}{l}
\displaystyle\mathbb{\varphi }(X^{\varepsilon }(T))-\mathbb{\varphi }(\bar{X}%
(T)) \\ 
\displaystyle=\int_{0}^{1}\mathbb{\varphi }_{x}\left( \theta ,\varepsilon
\right) )d\theta \left( X^{\varepsilon }(T)-\bar{X}(T)\right) \\ 
\displaystyle=\varepsilon \int_{0}^{1}\mathbb{\varphi }_{x}\left( \theta
,\varepsilon \right) d\theta \bigl(X_{1}(T)+\tilde{X}\mathbb{^{\varepsilon }}%
\left( T\right) \bigr) \\ 
\displaystyle=\varepsilon \mathbb{\varphi }_{x}(\bar{X}(T))X_{1}(T)+%
\varepsilon \int_{0}^{1}\left[ \mathbb{\varphi }_{x}\left( \theta
,\varepsilon \right) -\mathbb{\varphi }_{x}(\bar{X}(T))\right] d\theta
X_{1}(T)+\varepsilon \int_{0}^{1}\mathbb{\varphi }_{x}\left( \theta
,\varepsilon \right) d\theta \tilde{X}\mathbb{^{\varepsilon }}\left(
T\right) ,%
\end{array}
\label{taylor_1}
\end{equation}%
where $\mathbb{\varphi }_{x}\left( \theta ,\varepsilon \right) $ is defined
as before. We now estimate the last two terms in the above formula. For each 
$N\in \mathbb{N},$ define%
\begin{equation*}
\begin{array}{c}
A_{N}=\left\{ |\tilde{X}\mathbb{^{\varepsilon }}(T)|+|X_{1}(T)|\leq
N\right\} .%
\end{array}%
\end{equation*}%
From (\ref{uni_bound}) and Proposition \ref{3.2}, we derive from (H4) and
(H5) that%
\begin{equation*}
\begin{array}{l}
\displaystyle\hat{\mathbb{E}}\biggl[\biggl\vert\int_{0}^{1}\left[ \mathbb{%
\varphi }_{x}\left( \theta ,\varepsilon \right) -\mathbb{\varphi }_{x}(\bar{X%
}(T))\right] d\theta X_{1}(T)\biggr\vert\biggr] \\ 
\leq \hat{\mathbb{E}}\left[ \left\vert \bar{w}\Bigl(\varepsilon \bigl\vert%
\tilde{X}\mathbb{^{\varepsilon }}(T)+X_{1}(T)\bigr\vert\Bigr)%
X_{1}(T)I_{A_{N}}\right\vert \right] +\hat{\mathbb{E}}\left[ \left\vert
\left( 1+|\bar{X}(T)|+|\tilde{X}\mathbb{^{\varepsilon }}(T)%
|+| X_{1}(T)| \right)
X_{1}(T)I_{A_{N}^{c}}\right\vert \right] \\ 
\leq \left\vert \bar{w}\left( \varepsilon N\right) \right\vert \hat{\mathbb{E%
}}\bigl[\left\vert X_{1}(T)\right\vert ^{2}\bigr]^{1/2}+\hat{\mathbb{E}}%
\Bigl[\Bigl(1+|\bar{X}(T)|+|\tilde{X}\mathbb{^{\varepsilon }}(T)%
|+| X_{1}(T)| \Bigr)^{2} \Bigr]^{1/2}\hat{\mathbb{E}%
}\bigl[\left\vert X_{1}(T)\right\vert ^{2}I_{A_{N}^{c}}\bigr]^{1/2} \\ 
\leq C\Bigl\{\left\vert \bar{w}\left( \varepsilon N\right) \right\vert +\hat{%
\mathbb{E}}\bigl[\left\vert X_{1}(T)\right\vert ^{2}I_{A_{N}^{c}}\bigr]^{1/2}%
\Bigr\}.%
\end{array}%
\end{equation*}%
Following similar arguments as in (\ref{Var_3})-(\ref{Var_6}), we obtain
from Theorem 52 in \cite{Denis2011function}that 
\begin{equation}
\underset{\varepsilon \rightarrow 0}{\lim }\hat{\mathbb{E}}\biggl[\biggl\vert%
\int_{0}^{1}\left[ \mathbb{\varphi }_{x}\left( \theta ,\varepsilon \right) -%
\mathbb{\varphi }_{x}(\bar{X}(T))\right] d\theta X_{1}(T)\biggr\vert\biggr]%
=0.  \label{taylor_5}
\end{equation}%
By H$\ddot{\mathrm{o}}$lder's inequality and Proposition \ref{Variation_th}, we have%
\begin{equation*}
\begin{array}{c}
\displaystyle\hat{\mathbb{E}}\left[ \left\vert \int_{0}^{1}\mathbb{\varphi }%
_{x}\left( \theta ,\varepsilon \right) d\theta \tilde{X}\mathbb{%
^{\varepsilon }}\left( T\right) \right\vert \right] \leq \hat{\mathbb{E}}%
\bigl[\bigl\vert\tilde{X}\mathbb{^{\varepsilon }}\left( T\right) \bigr\vert%
^{2}\bigr]^{1/2}\hat{\mathbb{E}}\Bigl[\Bigl(1+\left\vert \bar{X}%
(T)\right\vert +\left\vert X_{1}(T)\right\vert +\bigl\vert\tilde{X}\mathbb{%
^{\varepsilon }}\left( T\right) \bigr\vert\Bigr)^{2}\Bigr]^{1/2},%
\end{array}%
\end{equation*}%
which implies that 
\begin{equation}
\begin{array}{c}
\displaystyle\underset{\varepsilon \rightarrow 0}{\lim }\hat{\mathbb{E}}%
\left[ \left\vert \int_{0}^{1}\mathbb{\varphi }_{x}\left( \theta
,\varepsilon \right) d\theta \tilde{X}\mathbb{^{\varepsilon }}\left(
T\right) \right\vert \right] =0.%
\end{array}
\label{taylor_2}
\end{equation}%
Then, by (\ref{taylor_2}) and (\ref{taylor_5}), we conclude that 
\begin{equation}
\underset{\varepsilon \rightarrow 0}{\lim }\hat{\mathbb{E}}\left[ \left\vert
\int_{0}^{1}\left[ \mathbb{\varphi }_{x}\left( \theta ,\varepsilon \right) -%
\mathbb{\varphi }_{x}(\bar{X}(T))\right] d\theta X_{1}(T)+\int_{0}^{1}%
\mathbb{\varphi }_{x}\left( \theta ,\varepsilon \right) d\theta \tilde{X}%
\mathbb{^{\varepsilon }}\left( T\right) \right\vert \right] =0.
\label{taylor_6}
\end{equation}

Next, we consider the integral term of $g$. By a similar Taylor expansion,
we obtain%
\begin{equation}
\displaystyle\int_{0}^{T}\left[ g^{\varepsilon }(s)-\bar{g}(s)\right]
ds=\varepsilon \int_{0}^{T}\left[ \bar{g}_{x}(s)X_{1}(s)+\bar{g}_{u}(s)v(s)%
\right] ds+\varepsilon \tilde{R}_{g}\left( \varepsilon \right) ,
\label{taylor_3}
\end{equation}%
where%
\begin{equation*}
\begin{array}{l@{}l}
\tilde{R}_{g}\left( \varepsilon \right) ={}& \displaystyle\int_{0}^{T}%
\int_{0}^{1}\left[ g_{x}\left( \theta ,s\right) -\bar{g}_{x}(s)\right]
d\theta X_{1}(s)ds+\int_{0}^{T}\int_{0}^{1}\left[ g_{u}\left( \theta
,s\right) -\bar{g}_{u}(s)\right] d\theta v(s)ds \\ 
& {} \displaystyle+\int_{0}^{T}\int_{0}^{1}g_{x}\left( \theta
,s\right) d\theta \tilde{X}\mathbb{^{\varepsilon }}\left( s\right) ds.%
\end{array}%
\end{equation*}%
Here $g_{x}\left( \theta ,s\right) $ and $g_{u}\left( \theta ,s\right) $
follow the definitions stated previously. By analogous arguments to those in
in (\ref{Var_2}), one can easily derive from (H4) that%
\begin{equation}
\underset{\varepsilon \rightarrow 0}{\lim }\hat{\mathbb{E}}\bigl[\bigl\vert%
\tilde{R}_{g}\left( \varepsilon \right) \bigr\vert\bigr]=0.  \label{taylor_4}
\end{equation}%
Combining with (\ref{taylor_6}) and (\ref{taylor_4}), the proof is complete.
\end{proof}

To obtain the stochastic maximum principle, we need to establish the Variational inequality$.$
The following is the main result.

\begin{theorem}
\label{ineq_Var}Assume that (H1)-(H6) hold. Then, for each $u(\cdot )\in 
\mathcal{U}[0,T],$ there exists a $\tilde{P}\in \mathcal{\tilde{P}}\left( 
\bar{X},\bar{u}\right) $ such that 
\begin{equation}
\begin{array}{c}
\underset{u\in \mathcal{U}[0,T]}{\inf }E_{\tilde{P}}\left[ \Psi \left(
u\right) \right] \geq 0.%
\end{array}
\label{ineq_Var_1}
\end{equation}
\end{theorem}

\begin{proof}
\textbf{Step 1}. We first give the limit of $\underset{\varepsilon
\rightarrow 0}{\lim }\frac{J\left( u^{\varepsilon }(\cdot )\right) -J\left( 
\bar{u}(\cdot )\right) }{\varepsilon }.$ Set 
\begin{equation}
\begin{array}{c}
\mathcal{P}_{\varepsilon }^{\ast }\mathcal{=}\left\{ P\in \mathcal{P}\text{ }%
|\text{ }E_{P}\mathbb{[}\Gamma \left( \bar{X},\bar{u}\right) +\varepsilon
\Psi \left( u\right) ]\mathbb{=\hat{E}[}\Gamma \left( \bar{X},\bar{u}\right)
+\varepsilon \Psi \left( u\right) \mathbb{]}\right\} ,\text{ }0<\varepsilon
\leq 1.%
\end{array}
\label{ineq_Var_11}
\end{equation}%
Recalling Lemma \ref{Taylor}, we derive that for each $P^{\varepsilon }\in 
\mathcal{P}_{\varepsilon }^{\ast }$, 
\begin{equation}
\begin{array}{l@{}l}
\displaystyle\frac{J\left( u^{\varepsilon }(\cdot )\right) -J\left( \bar{u}%
(\cdot )\right) }{\varepsilon } &={} \displaystyle \frac{\hat{\mathbb{E}}\left[
\Gamma \left( X^{\varepsilon },u^{\varepsilon }\right) \right] -\hat{\mathbb{%
E}}\left[ \Gamma \left( \bar{X},\bar{u}\right) \right] }{\varepsilon } \\ 
& \displaystyle\leq \frac{\hat{\mathbb{E}}\left[ \Gamma \left( \bar{X},\bar{u%
}\right) +\varepsilon \Psi \left( u\right) \right] -\hat{\mathbb{E}}\left[
\Gamma \left( \bar{X},\bar{u}\right) \right] }{\varepsilon }+\frac{\hat{%
\mathbb{E}}\left[ \left\vert \mathcal{R}(\varepsilon )\right\vert \right] }{%
\varepsilon } \\ 
& \displaystyle\leq \frac{E_{P^{\varepsilon }}\left[ \Gamma \left( \bar{X},%
\bar{u}\right) +\varepsilon \Psi \left( u\right) \right] -E_{P^{\varepsilon
}}\left[ \Gamma \left( \bar{X},\bar{u}\right) \right] }{\varepsilon }+\frac{%
\hat{\mathbb{E}}\left[ \left\vert \mathcal{R}(\varepsilon )\right\vert %
\right] }{\varepsilon } \\ 
& \displaystyle=E_{P^{\varepsilon }}\left[ \Psi \left( u\right) \right]
+o(1).%
\end{array}
\label{ineq_Var_2}
\end{equation}%
Since $\mathcal{P}$ is weakly compact, \ we can choose $\varepsilon
_{n}\rightarrow 0\ $and $P^{\varepsilon _{n}}\in \mathcal{P}_{\varepsilon
_{n}}^{\ast }$ such that $\left\{ P^{\varepsilon _{n}}\right\} _{n\geq 1}$
converges weekly to $\bar{P}\in \mathcal{P}$. Moreover, note that $\Gamma
\left( \bar{X},\bar{u}\right) ,\Psi \left( u\right) \in L_{G}^{1}(\Omega
_{T}).$ Then by Proposition \ref{weekconverge}, we obtain%
\begin{equation}
\begin{array}{c}
\underset{\varepsilon _{n}\rightarrow 0}{\lim }E_{P^{\varepsilon _{n}}}\left[
\Gamma \left( \bar{X},\bar{u}\right) \right] =E_{\bar{P}}\left[ \Gamma
\left( \bar{X},\bar{u}\right) \right]%
\end{array}
\label{ineq_Var_3}
\end{equation}%
and 
\begin{equation}
\begin{array}{c}
\underset{\varepsilon _{n}\rightarrow 0}{\lim }E_{P^{\varepsilon _{n}}}\left[
\Psi \left( u\right) \right] =E_{\bar{P}}\left[ \Psi \left( u\right) \right]%
\end{array}
\label{phi_lim}
\end{equation}%
Next, we show that $\bar{P}\in \mathcal{\tilde{P}}\left( \bar{X},\bar{u}%
\right) ,$ i.e. $E_{\bar{P}}\left[ \Gamma \left( \bar{X},\bar{u}\right) %
\right] =\hat{\mathbb{E}}\left[ \Gamma \left( \bar{X},\bar{u}\right) \right]
.$

Since $\hat{\mathbb{E}}\left[ \left\vert \Psi \left( u\right) \right\vert %
\right] <\infty $, it is easy to see that 
\begin{equation}
\begin{array}{c}
\underset{\varepsilon _{n}\rightarrow 0}{\lim }\left\vert \hat{\mathbb{E}}%
\left[ \Gamma \left( \bar{X},\bar{u}\right) +\varepsilon _{n}\Psi \left(
u\right) \right] -\hat{\mathbb{E}}\left[ \Gamma \left( \bar{X},\bar{u}%
\right) \right] \right\vert \leq \underset{\varepsilon _{n}\rightarrow 0}{%
\lim }\varepsilon _{n}\hat{\mathbb{E}}\left[ \left\vert \Psi \left( u\right)
\right\vert \right] =0.%
\end{array}
\label{ineq_Var_5}
\end{equation}%
By the definition of $\mathcal{P}_{\varepsilon _{n}}^{\ast }$, we have 
\begin{equation*}
E_{P^{\varepsilon _{n}}}\mathbb{[}\Gamma \left( \bar{X},\bar{u}\right)
+\varepsilon _{n}\Psi \left( u\right) ]\mathbb{=\hat{E}[}\Gamma \left( \bar{X%
},\bar{u}\right) +\varepsilon _{n}\Psi \left( u\right) \mathbb{]}\text{, for
each }n\geq 1,
\end{equation*}%
which implies that 
\begin{equation}
\begin{array}{l}
\underset{\varepsilon _{n}\rightarrow 0}{\lim }\left\vert \hat{\mathbb{E}}%
\left[ \Gamma \left( \bar{X},\bar{u}\right) +\varepsilon _{n}\Psi \left(
u\right) \right] -E_{P^{\varepsilon _{n}}}\left[ \Gamma \left( \bar{X},\bar{u%
}\right) \right] \right\vert \\ 
=\underset{\varepsilon _{n}\rightarrow 0}{\lim }\left\vert E_{P^{\varepsilon
_{n}}}\left[ \Gamma \left( \bar{X},\bar{u}\right) +\varepsilon _{n}\Psi
\left( u\right) \right] -E_{P^{\varepsilon _{n}}}\left[ \Gamma \left( \bar{X}%
,\bar{u}\right) \right] \right\vert \leq \underset{\varepsilon
_{n}\rightarrow 0}{\lim }\varepsilon _{n}\hat{\mathbb{E}}\left[ \left\vert
\Psi \left( u\right) \right\vert \right] =0.%
\end{array}
\label{ineq_Var_6}
\end{equation}%
Note that%
\begin{equation}
\begin{array}{l@{}l}
\left\vert \hat{\mathbb{E}}\left[ \Gamma \left( \bar{X},\bar{u}\right) %
\right] -E_{\bar{P}}\left[ \Gamma \left( \bar{X},\bar{u}\right) \right]
\right\vert \leq  {}& \left\vert \hat{\mathbb{E}}\left[ \Gamma \left( \bar{X},%
\bar{u}\right) \right] -\hat{\mathbb{E}}\left[ \Gamma \left( \bar{X},\bar{u}%
\right) +\varepsilon _{n}\Psi \left( u\right) \right] \right\vert +\Bigl\vert%
\hat{\mathbb{E}}\left[ \Gamma \left( \bar{X},\bar{u}\right) +\varepsilon
_{n}\Psi \left( u\right) \right] \Bigr. \\ 
{}& \Bigl.-E_{P^{\varepsilon _{n}}}\left[ \Gamma \left( \bar{X},\bar{%
u}\right) \right] \Bigr\vert+\Bigl\vert E_{P^{\varepsilon _{n}}}\left[
\Gamma \left( \bar{X},\bar{u}\right) \right] -E_{\bar{P}}\left[ \Gamma
\left( \bar{X},\bar{u}\right) \right] \Bigr\vert.%
\end{array}
\label{ineq_Var_7}
\end{equation}%
Then, combining (\ref{ineq_Var_3})-(\ref{ineq_Var_7}), it holds that 
\begin{equation*}
\begin{array}{c}
E_{\bar{P}}\left[ \Gamma \left( \bar{X},\bar{u}\right) \right] =\hat{\mathbb{%
E}}\left[ \Gamma \left( \bar{X},\bar{u}\right) \right] ,\text{ i.e. }\bar{P}%
\in \mathcal{\tilde{P}}\left( \bar{X},\bar{u}\right) .%
\end{array}%
\end{equation*}%
Consequently, from (\ref{ineq_Var_2}) and (\ref{phi_lim}), we obtain%
\begin{equation}
\begin{array}{l}
\displaystyle\underset{\varepsilon \rightarrow 0}{\lim \sup }\frac{J\left(
u^{\varepsilon }(\cdot )\right) -J\left( \bar{u}(\cdot )\right) }{%
\varepsilon }=\underset{n\rightarrow \infty }{\lim }\frac{J\left(
u^{\varepsilon _{n}}(\cdot )\right) -J\left( \bar{u}(\cdot )\right) }{%
\varepsilon _{n}} \\ 
\text{ \ \ \ \ \ \ \ \ \ \ \ \ \ \ \ \ \ \ \ \ \ \ \ \ \ \ \ \ \ \ \ \ \ }%
\displaystyle\leq E_{\bar{P}}\left[ \Psi \left( u\right) \right] \\ 
\text{ \ \ \ \ \ \ \ \ \ \ \ \ \ \ \ \ \ \ \ \ \ \ \ \ \ \ \ \ \ \ \ \ \ }%
\leq \underset{P\in \mathcal{\tilde{P}}\left( \bar{X},\bar{u}\right) }{\sup }%
E_{P}\left[ \Psi \left( u\right) \right] .%
\end{array}
\label{ineq_Var_8}
\end{equation}

Moreover, for each $P\in \mathcal{\tilde{P}}\left( \bar{X},\bar{u}\right) ,$%
\begin{equation*}
\begin{array}{l@{}l}
\displaystyle\frac{J\left( u^{\varepsilon }(\cdot )\right) -J\left( \bar{u}%
(\cdot )\right) }{\varepsilon } & \displaystyle{}=\frac{\hat{\mathbb{E}}\left[
\Gamma \left( X^{\varepsilon },u^{\varepsilon }\right) \right] -\hat{\mathbb{%
E}}\left[ \Gamma \left( \bar{X},\bar{u}\right) \right] }{\varepsilon } \\ 
& \displaystyle {}\geq \frac{E_{P}\left[ \Gamma \left( X^{\varepsilon
},u^{\varepsilon }\right) \right] -E_{P}\left[ \Gamma \left( \bar{X},\bar{u}%
\right) \right] }{\varepsilon } \\ 
& \displaystyle{}=\frac{E_{P}\left[ \Gamma \left( \bar{X},\bar{u}\right)
+\varepsilon \Psi \left( u\right) \right] -E_{P^{\varepsilon }}\left[ \Gamma
\left( \bar{X},\bar{u}\right) \right] }{\varepsilon }+\frac{E_{P}\left[
\left\vert \mathcal{R}(\varepsilon )\right\vert \right] }{\varepsilon } \\ 
& \displaystyle{}=E_{P}\left[ \Psi \left( u\right) \right] +\frac{E_{P}\left[
\left\vert \mathcal{R}(\varepsilon )\right\vert \right] }{\varepsilon }.%
\end{array}%
\end{equation*}%
By Lemma \ref{Taylor}, we derive 
\begin{equation}
\begin{array}{c}
\displaystyle\underset{\varepsilon \rightarrow 0}{\lim }\frac{E_{P}\left[
\left\vert \mathcal{R}(\varepsilon )\right\vert \right] }{\varepsilon }\leq 
\underset{\varepsilon \rightarrow 0}{\lim }\frac{\hat{\mathbb{E}}\left[
\left\vert \mathcal{R}(\varepsilon )\right\vert \right] }{\varepsilon }=0.%
\end{array}
\label{ineq_Var6}
\end{equation}%
Since $P$ is arbitrary, 
\begin{equation}
\begin{array}{c}
\displaystyle\underset{\varepsilon \rightarrow 0}{\lim \inf }\frac{J\left(
u^{\varepsilon }(\cdot )\right) -J\left( \bar{u}(\cdot )\right) }{%
\varepsilon }\geq \underset{P\in \mathcal{\tilde{P}}\left( \bar{X},\bar{u}%
\right) }{\sup }E_{P}\left[ \Psi \left( u\right) \right] .%
\end{array}
\label{ineq_Var_9}
\end{equation}%
Thus, we obtain from (\ref{ineq_Var_8})-(\ref{ineq_Var_9}) that%
\begin{equation}
\begin{array}{c}
\displaystyle\underset{\varepsilon \rightarrow 0}{\lim }\frac{J\left(
u^{\varepsilon }(\cdot )\right) -J\left( \bar{u}(\cdot )\right) }{%
\varepsilon }=\underset{P\in \mathcal{\tilde{P}}\left( \bar{X},\bar{u}%
\right) }{\sup }E_{P}\left[ \Psi \left( u\right) \right] .%
\end{array}
\label{ineq_Var_10}
\end{equation}

\textbf{Step 2. }Now we give the variational inequality. For each $u(\cdot
)\in \mathcal{U}[0,T],$ It has been proved in Step 1 that 
\begin{equation*}
\begin{array}{c}
\displaystyle\underset{\varepsilon \rightarrow 0}{\lim }\frac{J\left(
u^{\varepsilon }(\cdot )\right) -J\left( \bar{u}(\cdot )\right) }{%
\varepsilon }=\underset{P\in \mathcal{\tilde{P}}\left( \bar{X},\bar{u}%
\right) }{\sup }E_{P}\left[ \Psi \left( u\right) \right] \geq 0,%
\end{array}%
\end{equation*}%
which indicates that 
\begin{equation*}
\begin{array}{c}
\underset{u\in \mathcal{U}[0,T]}{\inf }\underset{P\in \mathcal{\tilde{P}}%
\left( \bar{X},\bar{u}\right) }{\sup }E_{P}\left[ \Psi \left( u\right) %
\right] \geq 0.%
\end{array}%
\end{equation*}%
Note that $\Psi \left( u\right) $ is convex in $u$ and $\mathcal{\tilde{P}}%
\left( \bar{X},\bar{u}\right) $ is convex and weakly compact (see Lemma \ref%
{weakness}). By minimax theorem, we obtain 
\begin{equation*}
\begin{array}{c}
\underset{u\in \mathcal{U}[0,T]}{\inf }\underset{P\in \mathcal{\tilde{P}}%
\left( \bar{X},\bar{u}\right) }{\sup }E_{P}\left[ \Psi \left( u\right) %
\right] =\underset{P\in \mathcal{\tilde{P}}\left( \bar{X},\bar{u}\right) }{%
\sup }\underset{u\in \mathcal{U}[0,T]}{\inf }E_{P}\left[ \Psi \left(
u\right) \right] \geq 0.%
\end{array}%
\end{equation*}%
By the definition of the supremum, for each $\varepsilon >0,$ there exists $%
P^{\varepsilon }\in \mathcal{\tilde{P}}\left( \bar{X},\bar{u}\right) $ such
that 
\begin{equation}
\underset{u\in \mathcal{U}[0,T]}{\inf }E_{P^{\varepsilon }}\left[ \Psi
\left( u\right) \right] \geq -\varepsilon .  \label{ineq_Var11}
\end{equation}%
By the weak compactness of $\mathcal{\tilde{P}}\left( \bar{X},\bar{u}\right)
,$ we derive that there exists a $\tilde{P}\in \mathcal{\tilde{P}}\left( 
\bar{X},\bar{u}\right) $ such that $P^{\varepsilon _{n}}\rightarrow \tilde{P}
$ weakly as $\varepsilon _{n}\rightarrow 0$. Then for each $u(\cdot )\in 
\mathcal{U}[0,T],$ we obtain from (\ref{ineq_Var11}) that 
\begin{equation}
\begin{array}{c}
\underset{\varepsilon _{n}\rightarrow 0}{\lim }E_{P^{\varepsilon _{n}}}\left[
\Psi \left( u\right) \right] =E_{\tilde{P}}\left[ \Psi \left( u\right) %
\right] \geq 0.%
\end{array}%
\end{equation}%
Consequently, we get (\ref{ineq_Var_1}). The proof is complete.
\end{proof}

\subsection{Stochastic maximum principle}

In this subsection, we will give the Necessary conditions for optimality. To
begin with, we introduce several spaces under $\tilde{P}\in \mathcal{\tilde{P%
}}\left( \bar{X},\bar{u}\right) $ and the definition of $M$-solution to
Backward stochastic Volterra integral equations (BSVIE). Let $\mathcal{F}%
_{t}^{\tilde{P}}=\sigma \left( B_{s},0\leq s\leq t\right) \vee \mathcal{N}^{%
\tilde{P}}$ be the completed filtration under $\tilde{P},$ where $\mathcal{N}%
^{\tilde{P}}$ denotes the family of all $\tilde{P}$-null sets. For $%
k=n,n\times d$, set 
\begin{equation*}
\left. 
\begin{array}{l}
\displaystyle L_{\tilde{P},\mathcal{F}_{T}}^{2}(0,T;\mathbb{R}^{k}):=%
\biggl\{ X:[0,T]\times \Omega \rightarrow \mathbb{R}^{k}|X\text{ is }%
\mathcal{F}_{T}\text{-measurable and }E_{\tilde{P}}\biggl[ %
\int_{0}^{T}|X(s)|^{2}ds\biggr] <\infty \biggr\} , \\ 
\displaystyle M_{\tilde{P}}^{2}(0,T;\mathbb{R}^{k}):=\biggl\{ X:[0,T]\times
\Omega \rightarrow \mathbb{R}^{k}|X\text{ is a progressively measurable and }%
E_{\tilde{P}}\biggl[ \int_{0}^{T}|X(s)|^{2}ds\biggr] <\infty \biggr\} , \\ 
\displaystyle\mathcal{M}_{\tilde{P}}^{2,\perp }(0,T;\mathbb{R}^{k}):=\left\{
X:[0,T]\times \Omega \rightarrow \mathbb{R}^{k}|X\text{ is a square
integrable martingale that is orthogonal to }B\right\} , \\ 
\displaystyle L_{\tilde{P}}^{2}\left( 0,T;M_{\tilde{P}}^{2}(0,T;\mathbb{R}%
^{n\times d})\right) :=\biggl\{ Z:\left[ 0,T\right] ^{2}\times \Omega
\rightarrow \mathbb{R}^{n\times d}|\text{ for each }t\in \left[ 0,T\right] ,%
\text{ }Z\left( t,\cdot \right) \in M_{\tilde{P}}^{2}(0,T;\mathbb{R}%
^{n\times d})\biggr. \\ 
\displaystyle \text{ \ \ \ \ \ \ \ \ \ \ \ \ \ \ \ \ \ \ \ \ \ \ \ \ \ \ \ \
\ \ \ \ \ \ \ \ \ \ \ \ }\biggl. \text{and }E_{\tilde{P}}\biggl[ %
\int_{0}^{T}\int_{0}^{T}\left\vert Z\left( t,s\right) \right\vert ^{2}dsdt%
\biggr] < \infty \biggr\} .%
\end{array}%
\right.
\end{equation*}%
In the sequel, for $\left[ a,b\right] \subset \left[ 0,T\right] $, we denote
by $L_{\tilde{P},\mathcal{F}_{T}}^{2}(a,b;\mathbb{R}^{k}),M_{\tilde{P}%
}^{2}(a,b;\mathbb{R}^{k})$ and $\mathcal{M}_{\tilde{P}}^{2,\perp }(a,b;%
\mathbb{R}^{k})$ the corresponding spaces of the stochastic processes on the
interval $\left[ a,b\right] $.

\begin{definition}
\label{N_def}Let $\mathcal{N}_{\tilde{P}}^{2}\bigl(0,T;\mathcal{M}_{\tilde{P}%
}^{2,\perp }(0,T;\mathbb{R}^{n\times d})\bigr)$ denote the set of all
processes $N:\left[ 0,T\right] ^{2}\times \Omega \rightarrow \mathbb{R}%
^{n\times d}$ such that 
\begin{equation*}
N\left( t,s\right) =\bar{N}\left( t,s\right) I_{\Delta ^{\ast }\left[ 0,T%
\right] }\left( t,s\right) +\tilde{N}\left( t,s\right) I_{\Delta \left[ 0,T%
\right] }\left( t,s\right) .
\end{equation*}%
Moreover, for each fixed $t\in \left[ 0,T\right] ,$ the process $N$
satisfies the following conditions:
\end{definition}

(i) $\left( \bar{N}\left( t,s\right) \right) _{s\in \left[ t,T\right] }\in 
\mathcal{M}_{\tilde{P}}^{2,\perp }(t,T;\mathbb{R}^{n\times d})$ with $\bar{N}%
\left( t,t\right) =0;$

(ii) $\bigl( \tilde{N}\left( t,s\right) \bigr) _{s\in \left[ 0,t\right] }\in 
\mathcal{M}_{\tilde{P}}^{2,\perp }(0,t;\mathbb{R}^{n\times d})$ with $\tilde{%
N}\left( t,0\right) =0;$

(iii) $\displaystyle E_{\tilde{P}}\biggl[ \int_{0}^{T}\left\langle \bar{N}%
\left( t,\cdot \right) \right\rangle _{T}dt\biggr] <\infty .$\newline
For briefly, we denote%
\begin{equation*}
\left\{ 
\begin{array}{l}
\mathcal{S}_{\tilde{P}}^{2}\left( 0,T\right) =M_{\tilde{P}}^{2}(0,T;\mathbb{R%
}^{n})\times M_{\tilde{P}}^{2}(0,T;\mathbb{R}^{n\times d})\times \mathcal{M}%
_{\tilde{P}}^{2,\perp }(0,T;\mathbb{R}^{n}); \\ 
\mathcal{H}_{\tilde{P}}^{2}\left( 0,T\right) =M_{\tilde{P}}^{2}(0,T;\mathbb{R%
}^{n})\times L_{\tilde{P}}^{2}\bigl( 0,T;M_{\tilde{P}}^{2}(0,T;\mathbb{R}%
^{n\times d})\bigr) \times \mathcal{N}_{\tilde{P}}^{2}\bigl (0,T;\mathcal{M}%
_{\tilde{P}}^{2,\perp }(0,T;\mathbb{R}^{n\times d})\bigr) .%
\end{array}%
\right.
\end{equation*}

Note that since the $G$-Brownian motion $B$ under $\tilde{P}$ is only a
continuous martingale$,$ the martingale representation theorem may fail in
general. Hence, we need to introduce an additional term $N$ orthogonal to $%
B. $ Accordingly, the BSVIE under $\tilde{P}$ takes the following form:%
\begin{equation}
\begin{array}{c}
\displaystyle Y\left( t\right) =\xi \left( t\right) +\int_{t}^{T}f\left(
t,s,Y\left( s\right) ,Z\left( t,s\right) ,Z\left( s,t\right) \right) ds-%
\underset{i=1}{\overset{d}{\sum }}\int_{t}^{T}Z^{i}\left( t,s\right)
dB_{s}^{i}-\underset{i=1}{\overset{d}{\sum }}\int_{t}^{T}dN^{i}\left(
t,s\right) .%
\end{array}
\label{BSVIE_N^}
\end{equation}%
To ensure the wellposedness of (\ref{BSVIE_N^}), we impose the following
assumption:

(H*) Suppose that $\xi \left( \cdot \right) \in L_{\tilde{P},\mathcal{F}%
_{T}}^{2}(0,T;\mathbb{R}^{n}),$ $f:\Delta ^{\ast }\left[ 0,T\right] \times 
\mathbb{R}^{n}\times \mathbb{R}^{n\times d}\times \mathbb{R}^{n\times
d}\times \Omega \rightarrow \mathbb{R}^{n}$ is $\mathcal{B}(\Delta
^{\ast }\left[ 0,T\right] \times \mathbb{R}^{n}\times \mathbb{R}^{n\times
d}\times \mathbb{R}^{n\times d}) \otimes \mathcal{F}_{T}$-measurable
and the mapping $s\mapsto f\left( t,s,y,z_{1},z_{2}\right) $ is
progressively measurable for each fixed $\left( t,y,z_{1},z_{2}\right) .$
Moreover, it holds that 
\begin{equation}
\begin{array}{c}
\displaystyle E_{\tilde{P}}\biggl[ \int_{0}^{T}\biggl( \int_{t}^{T}\left%
\vert f\left( t,s,0,0\right) \right\vert ds\biggr) ^{2}dt\biggr] <\infty%
\end{array}%
\end{equation}%
and there exists a constant $L>0$ such that for each $\left( t,s\right) \in
\Delta ^{\ast }\left[ 0,T\right] ,$ $y,y^{\prime }\in \mathbb{R}^{n},$ and $%
z_{1},z_{1}^{\prime },z_{2},z_{2}^{\prime }\in \mathbb{R}^{n\times d},$%
\begin{equation}
\left\vert f\left( t,s,y,z_{1},z_{2}\right) -f\left( t,s,y^{\prime
},z_{1}^{\prime },z_{2}^{\prime }\right) \right\vert \leq L\left( \left\vert
y-y^{\prime }\right\vert +\left\vert z_{1}-z_{1}^{\prime }\right\vert
+\left\vert z_{2}-z_{2}^{\prime }\right\vert \right).
\end{equation}

From \cite{yong2008Wellposedness}, we now proceed to give the definition of the $M$-solution
adapted to our setting under $\tilde{P}$.

\begin{definition}
\label{M_jie}Let (H*) hold. A triplet of processes $(Y,Z,N)$ is called an
adapted $M$-solution to BSVIE under $\tilde{P}$ if it satisfies the
following properties:
\end{definition}

(a) $(Y,Z,N)\in \mathcal{H}_{\tilde{P}}^{2}\left( 0,T\right) ,$

(b) $(Y,Z,N)$ satisfies%
\begin{equation}
\begin{array}{c}
\displaystyle Y\left( t\right) =\xi \left( t\right) +\int_{t}^{T}f\left(
t,s,Y\left( s\right) ,Z\left( t,s\right) ,Z\left( s,t\right) \right) ds-%
\underset{i=1}{\overset{d}{\sum }}\int_{t}^{T}Z^{i}\left( t,s\right)
dB_{s}^{i}-\underset{i=1}{\overset{d}{\sum }}\int_{t}^{T}dN^{i}\left(
t,s\right) ;%
\end{array}%
\end{equation}

(c) It holds that%
\begin{equation}
\begin{array}{c}
\displaystyle Y\left( t\right) =E_{\tilde{P}}\left[ Y\left( t\right) \right]
+\underset{i=1}{\overset{d}{\sum }}\int_{0}^{t}Z^{i}\left( t,s\right)
dB_{s}^{i}+\underset{i=1}{\overset{d}{\sum }}\int_{0}^{t}dN^{i}\left(
t,s\right) .%
\end{array}%
\end{equation}%
By martingale representation, the adapted $M$-solution determines the values
of $Z\left( t,s\right) $ and $N\left( t,s\right) $ for each $(t,s)\in \Delta %
\left[ 0,T\right] .$

\begin{remark}
By Proposition 3.12 in \cite{popier2021_N} and (H*), the BSVIE (\ref{BSVIE_N^}) admits a
unique $M$-solution $(Y,Z,N)\in \mathcal{H}_{\tilde{P}}^{2}\left( 0,T\right)
.$ The conditions $\bar{N}\left( t,t\right) =0$ and $\tilde{N}\left(
t,0\right) =0$ in Definition \ref{N_def} are imposed to ensure the
uniqueness of the $M$-solution to BSVIE. More precisely, without the
condition $\bar{N}\left( t,t\right) =0,$ the process%
\begin{equation}
\hat{N}\left( t,s\right) :=\bar{N}\left( t,s\right) +c,\text{ for any
constant }c\neq 0,
\end{equation}%
would also satisfy the same martingale and orthogonality conditions. Since $%
\int_{t}^{T}d\bar{N}\left( t,s\right) =\int_{t}^{T}d\hat{N}\left( t,s\right)
,$ it yields that both $\left( Y,Z,\bar{N}\right) $ and $(Y,Z,\hat{N})$
satisfy condition (b) in Definition \ref{M_jie}, violating the uniqueness of
the M-solution. The case for $N\left( t,0\right) =0$ is similar.
\end{remark}

\begin{remark}
\label{<B>}Under $\tilde{P}$, it holds that $d\left\langle
B^{i},B^{j}\right\rangle _{s}=\gamma ^{ij}\left( s\right) ds$, where $B$ is
a $d$-dimensional $G$-Brownian motion and $\left( \gamma ^{ij}\left(
s\right) \right) _{1\leq i,j\leq d}=\Sigma \left( s\right) .$ Further
details can be found in (\ref{G_def}).
\end{remark}

Next, we establish the duality principle of linear stochastic integral
equations.

\begin{proposition}
\label{duality principle}Let $\alpha \left( \cdot \right) \in \tilde{M}_{%
\tilde{P}}^{2}(0,T;\mathbb{R}^{n})$ be mean-square\ continuous and $\beta
\left( \cdot \right) \in L_{\tilde{P},\mathcal{F}_{T}}^{2}(0,T;\mathbb{R}%
^{n}).$ Suppose that $B_{1},B_{2},B_{3}$ satisfy (H1), (H3) and the mappings 
$B_{1},B_{2},B_{3}:\left[ 0,T\right] ^{2}\times \Omega \mapsto \mathbb{R}%
^{n\times n}$ are uniformly bounded. Let $\hat{X}\in \tilde{M}_{\tilde{P}%
}^{2}(0,T;\mathbb{R}^{n})$ be the solution of $G$-SVIE under $\tilde{P}:$ 
\begin{equation}
\begin{array}{c}
\displaystyle\hat{X}\left( t\right) =\alpha \left( t\right)
+\int_{0}^{t}B_{1}(t,s)\hat{X}(s)ds+\underset{i,j=1}{\overset{d}{\sum }}%
\int_{0}^{t}B_{2}^{ij}(t,s)\hat{X}(s)d\left\langle B^{i},B^{j}\right\rangle
_{s}+\underset{i=1}{\overset{d}{\sum }}\int_{0}^{t}B_{3}^{i}(t,s)\hat{X}%
(s)dB_{s}^{i},\text{ }t\in \left[ 0,T\right] ,%
\end{array}
\label{duality_x}
\end{equation}%
and $(\hat{Y},\hat{Z},\hat{N})\in \mathcal{H}_{\tilde{P}}^{2}\left(
0,T\right) \mathcal{\ }$be the adapted $M$-solution to BSVIE under $\tilde{P}%
:$%
\begin{equation}
\begin{array}{l}
\displaystyle\hat{Y}\left( t\right) =\beta \left( t\right)
+\int_{t}^{T}B_{1}(s,t)^{T}\hat{Y}\left( s\right) ds+\underset{i,j=1}{%
\overset{d}{\sum }}\int_{t}^{T}\left[ \gamma ^{ij}\left( t\right)
B_{2}^{ij}(s,t)^{T}\hat{Y}\left( s\right) +\gamma ^{ij}\left( t\right)
B_{3}^{i}(s,t)^{T}\hat{Z}^{j}\left( s,t\right) \right] ds \\ 
\displaystyle\text{ \ \ \ \ \ \ \ \ \ }-\underset{i=1}{\overset{d}{\sum }}%
\int_{t}^{T}\hat{Z}^{i}\left( t,s\right) dB_{s}^{i}-\underset{i=1}{\overset{d%
}{\sum }}\int_{t}^{T}d\hat{N}^{i}\left( t,s\right) .%
\end{array}
\label{duality_y}
\end{equation}%
Then the following relations holds: 
\begin{equation}
\begin{array}{c}
\displaystyle E_{\tilde{P}}\biggl[\int_{0}^{T}\bigl\langle\hat{X}\left(
t\right) ,\beta \left( t\right) \bigr\rangle dt\biggr]=E_{\tilde{P}}\biggl[%
\int_{0}^{T}\bigl\langle\alpha \left( t\right) ,\hat{Y}\left( t\right) %
\bigr\rangle dt\biggr].%
\end{array}%
\end{equation}
\end{proposition}

\begin{proof}
By Theorem \ref{non_ujie} and Proposition 3.12 in \cite{popier2021_N}, the $G$-SVIE(\ref%
{duality_x}) and BSVIE(\ref{duality_y}) admit a unique solution $\hat{X}\in 
\tilde{M}_{\tilde{P}}^{2}(0,T;\mathbb{R}^{n})$ and $(\hat{Y},\hat{Z},\hat{N}%
)\in \mathcal{H}_{\tilde{P}}^{2}\left( 0,T\right) $ respectively. For each $%
(t,s)\in \Delta \left[ 0,T\right] $ and $1\leq i,j\leq d,$ we denote%
\begin{equation*}
\begin{array}{c}
\displaystyle C^{ij}\left( t,s\right) =\frac{1}{d^{2}}%
B_{1}(t,s)+B_{2}^{ij}(t,s)\gamma ^{ij}\left( s\right) ,%
\end{array}%
\end{equation*}%
where $\gamma ^{ij}\left( s\right) \in \mathbb{R}^{1\times 1}$ and $\left(
\gamma ^{ij}\left( s\right) \right) =\Sigma \left( s\right) $ in (\ref{G_def}%
)$.$ By Remark \ref{<B>} and Eq.(\ref{duality_x}), we have%
\begin{equation}
\begin{array}{l}
\displaystyle E_{\tilde{P}}\biggl[\int_{0}^{T}\bigl\langle\alpha \left(
t\right) ,\hat{Y}\left( t\right) \bigr\rangle dt\biggr] \\ 
\displaystyle=E_{\tilde{P}}\biggl[\int_{0}^{T}\Bigl\langle\hat{X}\left(
t\right) -\underset{i,j=1}{\overset{d}{\sum }}\int_{0}^{t}C^{ij}\left(
t,s\right) \hat{X}(s)ds-\underset{i=1}{\overset{d}{\sum }}%
\int_{0}^{t}B_{3}^{i}(t,s)\hat{X}(s)dB_{s}^{i},\hat{Y}\left( t\right) %
\Bigr\rangle dt\biggr] \\ 
\displaystyle=E_{\tilde{P}}\biggl[\int_{0}^{T}\bigl\langle\hat{X}\left(
t\right) ,\hat{Y}\left( t\right) \bigr\rangle dt\biggr]-\underset{i,j=1}{%
\overset{d}{\sum }}E_{\tilde{P}}\biggl[\int_{0}^{T}\int_{0}^{t}\Bigl\langle%
\hat{X}(s),C^{ij}\left( t,s\right) ^{T}\hat{Y}\left( t\right) \Bigr\rangle %
dsdt\biggr] \\ 
\displaystyle\text{ \ \ }-E_{\tilde{P}}\biggl[\int_{0}^{T}\biggl\langle%
\underset{i=1}{\overset{d}{\sum }}\int_{0}^{t}B_{3}^{i}(t,s)\hat{X}%
(s)dB_{s}^{i},\hat{Y}\left( t\right) \biggr\rangle dt\biggr].%
\end{array}
\label{duality1}
\end{equation}%
From Proposition 2.17 of Section 3 in \cite{shu_K}and Definition \ref{M_jie}, we
derive%
\begin{equation}
\begin{array}{l}
\displaystyle E_{\tilde{P}}\biggl[\int_{0}^{T}\biggl\langle\underset{i=1}{%
\overset{d}{\sum }}\int_{0}^{t}B_{3}^{i}(t,s)\hat{X}(s)dB_{s}^{i},\hat{Y}%
\left( t\right) \biggr\rangle dt\biggr] \\ 
\displaystyle=E_{\tilde{P}}\biggl[\int_{0}^{T}\biggl\langle\underset{i=1}{%
\overset{d}{\sum }}\int_{0}^{t}B_{3}^{i}(t,s)\hat{X}(s)dB_{s}^{i},E_{\tilde{P%
}}\bigl[\hat{Y}\left( t\right) \bigr]+\underset{i=1}{\overset{d}{\sum }}%
\int_{0}^{t}\hat{Z}^{i}\left( t,s\right) dB_{s}^{i}-\underset{i=1}{\overset{d%
}{\sum }}\int_{0}^{t}d\hat{N}^{i}\left( t,s\right) \biggr\rangle dt\biggr]
\\ 
\displaystyle=\underset{i,j=1}{\overset{d}{\sum }}E_{\tilde{P}}\biggl[%
\int_{0}^{T}\int_{0}^{t}\left\langle B_{3}^{i}(t,s)\hat{X}(s),\hat{Z}%
^{j}\left( t,s\right) \right\rangle d\left\langle B^{i},B^{j}\right\rangle
_{s}dt\biggr] \\ 
\displaystyle=\underset{i,j=1}{\overset{d}{\sum }}E_{\tilde{P}}\biggl[%
\int_{0}^{T}\int_{0}^{t}\Bigl\langle\hat{X}(s),\gamma ^{ij}\left( s\right)
B_{3}^{i}(t,s)^{T}\hat{Z}^{j}\left( t,s\right) \Bigr \rangle dsdt\biggr].%
\end{array}
\label{duality2}
\end{equation}%
Combining (\ref{duality1}) with (\ref{duality2}), it holds that%
\begin{equation}
\begin{array}{l}
\displaystyle E_{\tilde{P}}\biggl[\int_{0}^{T}\bigl\langle\alpha \left(
t\right) ,\hat{Y}\left( t\right) \bigr\rangle dt\biggr] \\ 
\displaystyle=E_{\tilde{P}}\biggl[\int_{0}^{T}\bigl\langle\hat{X}\left(
t\right) ,\hat{Y}\left( t\right) \bigr\rangle dt\biggr]-\underset{i,j=1}{%
\overset{d}{\sum }}E_{\tilde{P}}\biggl[\int_{0}^{T}\int_{s}^{T}\Bigl\langle%
\hat{X}(s),C^{ij}\left( t,s\right) ^{T}\hat{Y}\left( t\right) \Bigr\rangle %
dtds\biggr] \\ 
\displaystyle\text{ \ \ }-\underset{i,j=1}{\overset{d}{\sum }}E_{\tilde{P}}%
\biggl[\int_{0}^{T}\int_{s}^{T}\left\langle \hat{X}(s),\gamma ^{ij}\left(
s\right) B_{3}^{i}(t,s)^{T}\hat{Z}^{j}\left( t,s\right) \right\rangle dtds%
\biggr] \\ 
\displaystyle=E_{\tilde{P}}\biggl[\int_{0}^{T}\bigl\langle\hat{X}\left(
t\right) ,\hat{Y}\left( t\right) \bigr\rangle dt\biggr]-\underset{i,j=1}{%
\overset{d}{\sum }}E_{\tilde{P}}\biggl[\int_{0}^{T}\int_{t}^{T}\left\langle 
\hat{X}(t),C^{ij}\left( s,t\right) ^{T}\hat{Y}\left( s\right) \right\rangle
dsdt\biggr] \\ 
\displaystyle\text{ \ \ }-\underset{i,j=1}{\overset{d}{\sum }}E_{\tilde{P}}%
\biggl[\int_{0}^{T}\int_{t}^{T}\left\langle \hat{X}(t),\gamma ^{ij}\left(
t\right) B_{3}^{i}(s,t)^{T}\hat{Z}^{j}\left( s,t\right) \right\rangle dsdt%
\biggr] \\ 
\displaystyle=E_{\tilde{P}}\biggl[\int_{0}^{T}\biggl\langle\hat{X}\left(
t\right) ,\hat{Y}\left( t\right) -\underset{i,j=1}{\overset{d}{\sum }}%
\int_{t}^{T}\Bigl[C^{ij}\left( s,t\right) ^{T}\hat{Y}\left( s\right) +\gamma
^{ij}\left( t\right) B_{3}^{i}(s,t)^{T}\hat{Z}^{j}\left( s,t\right) \Bigr]ds%
\biggr\rangle dt\biggr] \\ 
\displaystyle=E_{\tilde{P}}\biggl[\int_{0}^{T}\biggl\langle\hat{X}\left(
t\right) ,\beta \left( t\right) -\underset{i=1}{\overset{d}{\sum }}%
\int_{t}^{T}\hat{Z}^{i}\left( t,s\right) dB_{s}^{i}-\underset{i=1}{\overset{d%
}{\sum }}\int_{t}^{T}d\hat{N}^{i}\left( t,s\right) ds\biggr\rangle dt\biggr]
\\ 
\displaystyle=E_{\tilde{P}}\biggl[\int_{0}^{T}\biggl\langle\hat{X}\left(
t\right) ,E_{\tilde{P}}^{\mathcal{F}_{t}}\biggl[\beta \left( t\right) -%
\underset{i=1}{\overset{d}{\sum }}\int_{t}^{T}\hat{Z}^{i}\left( t,s\right)
dB_{s}^{i}-\underset{i=1}{\overset{d}{\sum }}\int_{t}^{T}d\hat{N}^{i}\left(
t,s\right) ds\biggr]\biggr\rangle dt\biggr] \\ 
\displaystyle=E_{\tilde{P}}\biggl [\int_{0}^{T}\bigl\langle\hat{X}\left(
t\right) ,\beta \left( t\right) \bigr\rangle dt\biggr].%
\end{array}
\label{duality3}
\end{equation}
\end{proof}

Consider the following adjoint system associated with the controlled $G$%
-SVIE (\ref{SVIE}) under $\tilde{P}$:%
\begin{empheq}[left=\empheqlbrace]{align}
	& X_{1}(t) = \zeta(t) + \underset{i,j=1}{\overset{d}{%
			\sum }}\int_{0}^{t}\bar{A}_{x}^{ij}(t,s)X_{1}(s)ds +  \underset{i=1}{\overset{d}{%
			\sum }}\int_{0}^{t}\bar{\sigma}_{x}^{i}(t,s)X_{1}(s)dB_{s}^{i}, \label{eq:X1} \\
	& \mu(t) = \varphi_{x}(\bar{X}(T)) - \underset{i=1}{\overset{d}{%
			\sum }} \int_{t}^{T}\eta^{i}(s)dB_{s}^{i} - \underset{i=1}{\overset{d}{%
			\sum }} \int_{t}^{T}dM^{i}(s), \label{eq:mu} \\
	& \begin{aligned}
		Y(t) =& \bar{g}_{x}(t) +\underset{i,j=1}{\overset{d}{%
				\sum }}\Bigl[ \bar{A}_{x}^{ij}(T,t)^{T}\mathbb{\varphi }_{x}(\bar{X}%
		(T))+\gamma ^{ij}\left( t\right) \bar{\sigma}_{x}^{i}(T,t)^{T}\eta
		^{j}\left( t\right) \Bigr] \\
		&+\underset{i,j=1}{\overset{d}{%
				\sum }} \int_{t}^{T}\Bigl[ \bar{A}_{x}^{ij}(s,t)Y(s) + \gamma ^{ij}\left( t\right) \bar{\sigma}_{x}^{i}(s,t)Z^{j}(s,t) \Bigr] ds - \underset{i=1}{\overset{d}{%
				\sum }} \int_{t}^{T}Z^{i}(t,s)dB^{i}_{s}  \\
			&-\underset{i=1}{\overset{d}{%
				\sum }} \int_{t}^{T}dN^{i}(t,s),\
	\end{aligned} \label{eq:Y}
\end{empheq}where 
\begin{equation}
\begin{array}{c}
\displaystyle\bar{A}_{x}^{ij}(t,s)=\frac{1}{d^{2}}\bar{b}_{x}(t,s)+\bar{h}%
_{x}^{ij}(t,s)\gamma ^{ij}\left( s\right) ,\bar{A}_{u}^{ij}(t,s)=\frac{1}{%
d^{2}}\bar{b}_{u}(t,s)+\bar{h}_{u}^{ij}(t,s)\gamma ^{ij}\left( s\right)%
\end{array}
\label{A_Def}
\end{equation}
and 
\begin{equation}
\begin{array}{c}
\displaystyle\zeta \left( t\right) =\underset{i,j=1}{\overset{d}{\sum }}%
\int_{0}^{t}\bar{A}_{u}^{ij}(t,s)\left( u(s)-\bar{u}(s)\right) ds+\underset{%
i=1}{\overset{d}{\sum }}\int_{0}^{t}\bar{\sigma}_{u}^{i}(t,s)\left( u(s)-%
\bar{u}(s)\right) dB_{s}^{i}.%
\end{array}%
\end{equation}%
See (\ref{cofficient_def}) for more details.

Note that Eq.(\ref{eq:X1}) is the variational equation (see (\ref{Variation Eq})), which admits a
unique solution $X_{1}(\cdot )\in \tilde{M}_{\tilde{P}}^{2}(0,T;\mathbb{R}%
^{n}).$ According to El Karoui and Huang \cite{Karoui1997bsde_N}, the
BSDE (\ref{eq:mu}) has a unique $\left( \mu ,\eta ,M\right) \in \mathcal{S}_{%
\tilde{P}}^{2}\left( 0,T\right) $. Moreover, by Proposition 3.12 in \cite{popier2021_N},
BSVIE (\ref{eq:Y}) admits a unique solution $\left( Y,Z,N\right) \in 
\mathcal{H}_{\tilde{P}}^{2}\left( 0,T\right) .$ Then we introduce the main
result in this Section. Define the Hamiltonian function:%
\begin{equation}
\begin{array}{l@{}l}
H(t,x,Y\left( t\right) ,Z\left( \cdot ,t\right) ,u)= {}& g\left( t,x,u\right) +%
\underset{i,j=1}{\overset{d}{\sum }}A^{ij}(T,t,x,u)^{T}E_{\tilde{P}}^{%
\mathcal{F}_{t}}\left[ \mathbb{\varphi }_{x}(\bar{X}(T))\right] +\underset{%
i,j=1}{\overset{d}{\sum }}\gamma ^{ij}\left( t\right) \sigma
^{i}(T,t,x,u)^{T}\eta ^{j}(t) \\ 
{}& \displaystyle+\underset{i,j=1}{\overset{d}{\sum }}E_{\tilde{P}}^{\mathcal{F%
}_{t}}\biggl[\int_{t}^{T}A^{ij}(s,t,x,u)^{T}Y(s)+\gamma ^{ij}\left( t\right)
\sigma ^{i}(s,t,x,u)^{T}Z^{j}\left( s,t\right) ds\biggr],%
\end{array}
\label{Hamiltonian function.}
\end{equation}%
where $\eta ^{j}(\cdot ),Y\left( \cdot \right) ,Z\left( \cdot ,\cdot \right) 
$ is determined in (\ref{eq:mu}),(\ref{eq:Y}) and 
\begin{equation}
A^{ij}(t,s,x,u)=\frac{1}{d^{2}}b(t,s,x,u)+h^{ij}(t,s,x,u)\gamma ^{ij}\left(
s\right) .  \label{Axu_Def}
\end{equation}
We now state the following stochastic maximum principle.

\begin{theorem}
\label{SMP}Suppose that (H1)-(H6) hold. Let $(\bar{X}(\cdot ),$ $\bar{u}%
(\cdot ))$ be the optimal 2-tuple of Problem (\ref{SG}). Then for each $%
u(\cdot )\in \mathcal{U}[0,T],$ there exists a $\tilde{P}\in \mathcal{\tilde{%
P}}\left( \bar{X},\bar{u}\right) $ such that%
\begin{equation}
\Bigl\langle H_{u}(t,\bar{X}\left( t\right) ,Y\left( t\right) ,Z\left( \cdot
,t\right) ,\bar{u}\left( t\right) ),\left( u(t)-\bar{u}(t)\right) %
\Bigr\rangle\geq 0,\text{a.e. }t\in \left[ 0,T\right] \text{, }\tilde{P}%
\text{-}a.s.
\end{equation}%
where $\left( Y,Z,N\right) \in \mathcal{H}_{\tilde{P}}^{2}\left( 0,T\right) $
is the solution of the adjoint equation (\ref{eq:Y}) under $\tilde{P}$ and
the Hamiltonian function $H$ is defined in (\ref{Hamiltonian function.}).
\end{theorem}

\begin{proof}
Our proof is based on the variational inequality given in Theorem \ref%
{ineq_Var}.

\textbf{Step 1.} We first compute $E_{\tilde{P}}\left[ \mathbb{\varphi }_{x}(%
\bar{X}(T))X_{1}(T)\right] .$ By (\ref{eq:mu}), we have 
\begin{equation}
\begin{array}{c}
\displaystyle\mathbb{\varphi }_{x}(\bar{X}(T))=\mu \left( T\right) =\mu
\left( 0\right) +\underset{i=1}{\overset{d}{\sum }}\int_{0}^{T}\eta
^{i}\left( s\right) dB_{s}^{i}+\underset{i=1}{\overset{d}{\sum }}%
\int_{0}^{T}dM^{i}\left( s\right) .%
\end{array}
\label{SMP5}
\end{equation}%
Since $N\in \mathcal{M}_{\tilde{P}}^{2,\perp }(0,T;\mathbb{R}^{n})$, it
follows from Proposition 2.17 of Section 3 in \cite{shu_K} that%
\begin{equation}
\begin{array}{l}
\displaystyle E_{\tilde{P}}\left[ \mathbb{\varphi }_{x}(\bar{X}(T))X_{1}(T)%
\right] \\ 
\displaystyle=E_{\tilde{P}}\biggl[\mathbb{\varphi }_{x}(\bar{X}(T))\underset{%
i,j=1}{\overset{d}{\sum }}\int_{0}^{T}\Bigl[\bar{A}_{x}^{ij}(T,s)X_{1}(s)+%
\bar{A}_{u}^{ij}(T,s)\bigl(u(s)-\bar{u}(s)\bigr)\Bigr]ds\biggr] \\ 
\displaystyle\text{ \ }+E_{\tilde{P}}\biggl[\biggl(\mu \left( 0\right) +%
\underset{i=1}{\overset{d}{\sum }}\int_{0}^{T}\eta ^{i}\left( s\right)
dB_{s}^{i}+\underset{i=1}{\overset{d}{\sum }}\int_{0}^{T}dM^{i}\left(
s\right) \biggr)\underset{i=1}{\overset{d}{\sum }}\int_{0}^{T}\Bigl[\bar{%
\sigma}_{x}^{i}(T,s)X_{1}(s)+\bar{\sigma}_{u}^{i}(T,s)\left( u(s)-\bar{u}%
(s)\right) \Bigr]dB_{s}^{i}\biggr] \\ 
\displaystyle=\underset{i,j=1}{\overset{d}{\sum }}E_{\tilde{P}}\biggl[%
\int_{0}^{T}\Bigl\langle\bar{A}_{x}^{ij}(T,s)^{T}\mathbb{\varphi }_{x}(\bar{X%
}(T))+\gamma ^{ij}\left( s\right) \bar{\sigma}_{x}^{i}(T,s)^{T}\eta
^{j}(s),X_{1}(s)\Bigr\rangle ds\biggr] \\ 
\displaystyle\text{ \ }+\underset{i,j=1}{\overset{d}{\sum }}E_{\tilde{P}}%
\biggl[\int_{0}^{T}\Bigl\langle\bar{A}_{u}^{ij}(T,s)^{T}\mathbb{\varphi }%
_{x}(\bar{X}(T))+\gamma ^{ij}\left( s\right) \bar{\sigma}_{u}^{i}(T,s)^{T}%
\eta ^{j}(s),\left( u(s)-\bar{u}(s)\right) \Bigr\rangle ds\biggr],%
\end{array}
\label{duality2.}
\end{equation}%
where $\bar{A}_{x}^{ij}(T,s)^{T},\bar{A}_{u}^{ij}(T,s)^{T},\bar{\sigma}%
_{x}^{i}(T,s),\bar{\sigma}_{u}^{i}(T,s)^{T}$ are defined in (\ref%
{cofficient_def})$.$ Consequently,%
\begin{equation}
\begin{array}{l}
\displaystyle E_{\tilde{P}}\left[ \Psi \left( u\right) \right] \\ 
\displaystyle=E_{\tilde{P}}\biggl[\mathbb{\varphi }_{x}(\bar{X}%
(T))X_{1}(T)+\int_{0}^{T}\Bigl[\bar{g}_{x}(s)X_{1}(s)+\bar{g}_{u}(s)\left(
u(s)-\bar{u}(s)\right) \Bigr]ds\biggr] \\ 
\displaystyle=E_{\tilde{P}}\biggl[\int_{0}^{T}\Bigl\langle\bar{g}_{x}(s)+%
\underset{i,j=1}{\overset{d}{\sum }}\Bigl[\bar{A}_{x}^{ij}(T,s)^{T}\mathbb{%
\varphi }_{x}(\bar{X}(T))+\gamma ^{ij}\left( s\right) \sigma
_{x}^{i}(T,s)^{T}\eta ^{j}(s)\Bigr],X_{1}(s)\Bigr\rangle ds\biggl] \\ 
\displaystyle\text{ \ \ }+E_{\tilde{P}}\biggl[\int_{0}^{T}\Bigl\langle\bar{g}%
_{u}(s)+\underset{i,j=1}{\overset{d}{\sum }}\Bigl[\bar{A}_{u}^{ij}(T,s)^{T}%
\mathbb{\varphi }_{x}(\bar{X}(T))+\gamma ^{ij}\left( s\right) \sigma
_{u}^{i}(T,s)^{T}\eta ^{j}(s)\Bigr],\left( u(s)-\bar{u}(s)\right) %
\Bigr\rangle ds\biggr].%
\end{array}
\label{SMP2}
\end{equation}

\textbf{Step 2. }In the following, we need to use the duality principles to
eliminate $X_{1}(\cdot ).$ Following a similar argument as in (\ref%
{duality2.}) and (\ref{duality3}) , we deduce from Proposition \ref{duality
principle} that%
\begin{equation}
\begin{array}{l}
\displaystyle E_{\tilde{P}}\biggl[\int_{0}^{T}\Bigl\langle\bar{g}_{x}(s)+%
\underset{i,j=1}{\overset{d}{\sum }}\Bigl[\bar{A}_{x}^{ij}(T,s)^{T}\mathbb{%
\varphi }_{x}(\bar{X}(T))+\gamma ^{ij}\left( s\right) \sigma
_{x}^{i}(T,s)^{T}\eta ^{j}(s)\Bigr],X_{1}(s)\Bigr\rangle ds\biggr] \\ 
\displaystyle=E_{\tilde{P}}\biggl[\int_{0}^{T}\left\langle \zeta \left(
t\right) ,Y(t)\right\rangle dt\biggr] \\ 
\displaystyle=E_{\tilde{P}}\biggl[\int_{0}^{T}\biggl\langle\underset{i=1}{%
\overset{d}{\sum }}\int_{0}^{t}\bar{\sigma}_{u}^{i}(t,s)\left( u(s)-\bar{u}%
(s)\right) dB_{s}^{i},E_{\tilde{P}}\left[ Y\left( t\right) \right] +\underset%
{i=1}{\overset{d}{\sum }}\int_{0}^{t}Z^{i}\left( t,s\right) dB_{s}^{i}+%
\underset{i=1}{\overset{d}{\sum }}\int_{0}^{t}dN^{i}\left( t,s\right) %
\biggr\rangle dt\biggr] \\ 
\displaystyle\text{ \ \ }+E_{\tilde{P}}\biggl[\int_{0}^{T}\biggl\langle%
\underset{i,j=1}{\overset{d}{\sum }}\int_{0}^{t}\bar{A}_{u}^{ij}(t,s)\left(
u(s)-\bar{u}(s)\right) ds,Y(t)\biggr\rangle dt\biggr] \\ 
\displaystyle=\underset{i,j=1}{\overset{d}{\sum }}E_{\tilde{P}}\biggl[%
\int_{0}^{T}\int_{0}^{t}\Bigl\langle\bar{\sigma}_{u}^{i}(t,s)^{T}Z^{j}\left(
t,s\right) ,\left( u(s)-\bar{u}(s)\right) \Bigr\rangle d\left\langle
B^{i},B^{j}\right\rangle _{s}dt\biggr] \\ 
\displaystyle\text{ \ \ }+\underset{i,j=1}{\overset{d}{\sum }}E_{\tilde{P}}%
\biggl[\int_{0}^{T}\int_{0}^{t}\Bigl\langle\bar{A}_{u}^{ij}(t,s)^{T}Y(t),%
\left( u(s)-\bar{u}(s)\right) \Bigr\rangle dsdt\biggr] \\ 
\displaystyle=\underset{i,j=1}{\overset{d}{\sum }}E_{\tilde{P}}\biggl[%
\int_{0}^{T}\int_{s}^{T}\Bigl\langle\gamma ^{ij}\left( s\right) \bar{\sigma}%
_{u}^{i}(t,s)^{T}Z^{j}\left( t,s\right) ,\left( u(s)-\bar{u}(s)\right) %
\Bigr\rangle dtds\biggr] \\ 
\displaystyle\text{ \ \ }+\underset{i,j=1}{\overset{d}{\sum }}E_{\tilde{P}}%
\biggl[\int_{0}^{T}\int_{s}^{T}\Bigl\langle\bar{A}_{u}^{ij}(t,s)^{T}Y(t),%
\left( u(s)-\bar{u}(s)\right) \Bigr\rangle dtds\biggr] \\ 
\displaystyle=\underset{i,j=1}{\overset{d}{\sum }}E_{\tilde{P}}\biggl[%
\int_{0}^{T}\int_{t}^{T}\Bigl\langle\Bigl[\bar{A}_{u}^{ij}(s,t)^{T}Y(s)+%
\gamma ^{ij}\left( t\right) \bar{\sigma}_{u}^{i}(s,t)^{T}Z^{j}\left(
s,t\right) \Bigr],\left( u(t)-\bar{u}(t)\right) \Bigr\rangle dsdt\biggr].%
\end{array}
\label{SMP3}
\end{equation}%
Thus, combining (\ref{SMP2}) with (\ref{SMP3}), we have for each $u(\cdot
)\in \mathcal{U}[0,T].$%
\begin{equation}
\begin{array}{r@{}l}
\displaystyle E_{\tilde{P}}\left[ \Psi \left( u\right) \right] {}=& E_{\tilde{P}}%
\biggl[\int_{0}^{T}\Bigl\langle\bar{g}_{u}(t)+\underset{i,j=1}{\overset{d}{%
\sum }}\left[ \bar{A}_{u}^{ij}(T,t)^{T}\mathbb{\varphi }_{x}(\bar{X}%
(T))+\gamma ^{ij}\left( t\right) \sigma _{u}^{i}(T,t)^{T}\eta ^{j}(t)\right]
,\left( u(t)-\bar{u}(t)\right) \Bigr\rangle dt\biggr] \\ 
{} & {}\displaystyle+E_{\tilde{P}}\biggl[%
\int_{0}^{T}\Bigl\langle\underset{i,j=1}{\overset{d}{\sum }}\int_{t}^{T}%
\Bigl[\bar{A}_{u}^{ij}(s,t)^{T}Y(s)+\gamma ^{ij}\left( t\right) \bar{\sigma}%
_{u}^{i}(s,t)^{T}Z^{j}\left( s,t\right) \Bigr]ds,\left( u(t)-\bar{u}%
(t)\right) \Bigr\rangle dt\biggr] \\ 
{}=&\displaystyle E_{\tilde{P}}\biggl[\int_{0}^{T}%
\Bigl\langle H_{u}(t,\bar{X}\left( t\right) ,Y\left( t\right) ,Z\left( \cdot
,t\right) ,\bar{u}\left( t\right) ),\left( u(t)-\bar{u}(t)\right) %
\Bigr\rangle dt\biggr]\geq 0.%
\end{array}
\label{SMP4}
\end{equation}

\textbf{Step 3. }We aim to verify that for any each $u(\cdot )\in \mathcal{U}%
[0,T],$ 
\begin{equation*}
\begin{array}{c}
\Bigl\langle H_{u}\bigl(t,\bar{X}\left( t\right) ,Y\left( t\right) ,Z\left(
\cdot ,t\right) ,\bar{u}\left( t\right) \bigr),\left( u(t)-\bar{u}(t)\right) %
\Bigr\rangle\geq 0,\text{ }a.e.\text{ }t\in \left[ 0,T\right] \text{, }%
\tilde{P}\text{-}a.s.%
\end{array}%
\end{equation*}%
The proof proceeds via contradiction. Set 
\begin{equation*}
\begin{array}{c}
A=\Bigl\{\left( t,w\right) \Bigl \vert\Bigl\langle H_{u}(t,\bar{X}\left(
t\right) ,Y\left( t\right) ,Z\left( \cdot ,t\right) ,\bar{u}\left( t\right)
),\left( u_{0}(t)-\bar{u}(t)\right) \Bigr\rangle<0\Bigr\}%
\end{array}%
\end{equation*}%
and $\hat{v}\left( \cdot \right) =u_{0}\left( \cdot \right) I_{A}+\bar{u}%
\left( \cdot \right) I_{A^{c}}\in \mathcal{U}[0,T].$ Assume that there
exists $u_{0}(\cdot )\in \mathcal{U}[0,T]$ such that%
\begin{equation*}
\begin{array}{c}
E_{\tilde{P}}\left[ \int_{0}^{T}I_{A}\left( t,w\right) dt\right] >0.%
\end{array}%
\end{equation*}%
It follows that 
\begin{equation*}
\begin{array}{l}
\displaystyle E_{\tilde{P}}\biggl[\int_{0}^{T}\Bigl\langle H_{u}(t,\bar{X}%
\left( t\right) ,Y\left( t\right) ,Z\left( \cdot ,t\right) ,\bar{u}\left(
t\right) ),\left( \hat{v}(t)-\bar{u}(t)\right) \Bigr\rangle dt\biggr] \\ 
\displaystyle=E_{\tilde{P}}\biggl[\int_{0}^{T}\Bigl\langle H_{u}(t,\bar{X}%
\left( t\right) ,Y\left( t\right) ,Z\left( \cdot ,t\right) ,\bar{u}\left(
t\right) ),\left( u_{0}(t)-\bar{u}(t)\right) I_{A}\Bigr\rangle dt\biggr]<0.%
\end{array}%
\end{equation*}%
This yields an obvious contradiction to (\ref{SMP4}), which completes the
proof.
\end{proof}

\section{Sufficient condition for optimality}

In this section, we give the sufficient condition for optimality.

\begin{theorem}
Let assumptions (H1)-(H6) hold. Suppose that $\bar{u}(\cdot )\in \mathcal{U}%
[0,T]$ and $\tilde{P}\in \mathcal{\tilde{P}}\left( \bar{X},\bar{u}\right) $
satisfy%
\begin{equation*}
\begin{array}{c}
\Bigl\langle H_{u}(t,\bar{X}\left( t\right) ,Y\left( t\right) ,Z\left( \cdot
,t\right) ,\bar{u}\left( t\right) ),\left( u(t)-\bar{u}(t)\right) %
\Bigr\rangle\geq 0,\text{ }\forall u\in U,\text{ }a.e.\text{ }t\in \left[ 0,T%
\right] ,\tilde{P}\text{-}a.s.,%
\end{array}%
\end{equation*}%
where $\bar{X}\left( \cdot \right) $ is the state process of $G$-SVIE (\ref%
{SVIE}) corresponding to $\bar{u}(\cdot )$ and $\left( Y,Z,N\right) $ is the
solution of the adjoint equation (\ref{eq:Y}) under $\tilde{P}.$ Moreover,
we assume that the Hamiltonian function $H$ (\ref{Hamiltonian function.}) is
convex in $x$,$u$ and $\varphi $ is convex in $x$,$u.$ Then $\left( \bar{X}%
(\cdot ),\bar{u}(\cdot )\right) $ is an optimal pair.
\end{theorem}

\begin{proof}
For each $u(\cdot )\in \mathcal{U}[0,T],$ let $X^{u}\left( \cdot \right) $
is the state process of $G$-SVIE (\ref{SVIE}) corresponding to $u(\cdot ).$
To simplify, we denote 
\begin{equation*}
H^{u}\left( t\right) =H(t,X^{u}\left( t\right) ,Y\left( t\right) ,Z\left(
\cdot ,t\right) ,u\left( t\right) )\text{ and }\bar{H}\left( t\right) =H(t,%
\bar{X}\left( t\right) ,Y\left( t\right) ,Z\left( \cdot ,t\right) ,\bar{u}%
\left( t\right) ).
\end{equation*}%
The partial derivatives $H_{x}^{u}\left( t\right) ,H_{u}^{u}\left( t\right) ,%
\bar{H}_{x}\left( t\right) ,\bar{H}_{u}\left( t\right) $ are defined
similarly. By (\ref{Hamiltonian function.})$,$ we have%
\begin{equation}
\begin{array}{l}
\displaystyle E_{\tilde{P}}\biggl[\int_{0}^{T}H^{u}\left( t\right) -\bar{H}%
\left( t\right) dt\biggr]=E_{\tilde{P}}\biggl[\int_{0}^{T}g\left(
t,X^{u}\left( t\right) ,u\left( t\right) \right) -g\bigl(t,\bar{X}\left(
t\right) ,\bar{u}\left( t\right) \bigr)dt\biggr] \\ 
\displaystyle\text{ \ \ \ \ \ \ \ \ \ \ \ \ \ \ \ \ \ \ \ \ \ \ \ \ \ \ \ \
\ \ \ \ \ \ \ \ \ }+\underset{i,j=1}{\overset{d}{\sum }}E_{\tilde{P}}\biggl[%
\int_{0}^{T}\Bigl[\delta A^{ij}(T,t)^{T}\mathbb{\varphi }_{x}(\bar{X}%
(T))+\gamma ^{ij}\left( t\right) \delta \sigma ^{i}(T,t)^{T}\eta ^{j}(t)%
\Bigr]dt\biggr] \\ 
\displaystyle\text{ \ \ \ \ \ \ \ \ \ \ \ \ \ \ \ \ \ \ \ \ \ \ \ \ \ \ \ \
\ \ \ \ \ \ \ \ \ }+\underset{i,j=1}{\overset{d}{\sum }}E_{\tilde{P}}\biggl[%
\int_{0}^{T}\int_{t}^{T}\Bigl[\delta A^{ij}(s,t)^{T}Y(s)+\gamma ^{ij}\left(
t\right) \delta \sigma ^{i}(s,t)^{T}Z^{j}\left( s,t\right) \Bigr]dsdt\biggr],%
\end{array}
\label{SF1}
\end{equation}%
where $A^{ij}(t,s,x,u)$ is defined in (\ref{Axu_Def}) and 
\begin{equation*}
\delta A^{ij}(t,s)^{T}=A^{ij}(t,s,X^{u}\left( t\right) ,u\left( t\right)
)^{T}-A^{ij}(t,s,\bar{X}\left( t\right) ,\bar{u}\left( t\right) )^{T},
\end{equation*}%
\begin{equation*}
\delta \sigma ^{i}(t,s)^{T}=\sigma ^{i}(t,s,X^{u}\left( t\right) ,u\left(
t\right) )^{T}-\sigma ^{i}(t,s,\bar{X}\left( t\right) ,\bar{u}\left(
t\right) )^{T},\text{ }t,s\in \left[ 0,T\right] ^{2}.
\end{equation*}%
We now compute the last two terms in (\ref{SF1}). Note that%
\begin{equation*}
X^{u}\left( t\right) -X\left( t\right) =\underset{i,j=1}{\overset{d}{\sum }}%
\int_{0}^{T}\delta A^{ij}(t,s)ds+\underset{i=1}{\overset{d}{\sum }}%
\int_{0}^{T}\delta \sigma ^{i}(t,s)dB_{s}^{i},\text{ }t\in \left[ 0,T\right]
.
\end{equation*}%
Since $\left( \mu ,\eta ,M\right) \in \mathcal{S}_{\tilde{P}}^{2}\left(
0,T\right) $, we obtain from (\ref{SMP5}) that%
\begin{equation}
\begin{array}{l}
\displaystyle\underset{i,j=1}{\overset{d}{\sum }}E_{\tilde{P}}\biggl[%
\int_{0}^{T}\delta A^{ij}(T,t)^{T}\mathbb{\varphi }_{x}(\bar{X}(T))+\gamma
^{ij}\left( t\right) \delta \sigma ^{i}(T,t)^{T}\eta ^{j}(t)dt\biggr] \\ 
\displaystyle=\underset{i,j=1}{\overset{d}{\sum }}E_{\tilde{P}}\biggl[%
\mathbb{\varphi }_{x}(\bar{X}(T))\int_{0}^{T}\delta A^{ij}(T,t)dt\biggr]+%
\underset{i,j=1}{\overset{d}{\sum }}E_{\tilde{P}}\biggl[\int_{0}^{T}\eta
^{j}(t)dB_{t}^{j}\int_{0}^{T}\delta \sigma ^{i}(T,t)dB_{t}^{i}\biggr] \\ 
\displaystyle=\underset{i,j=1}{\overset{d}{\sum }}E_{\tilde{P}}\biggl[%
\mathbb{\varphi }_{x}(\bar{X}(T))\int_{0}^{T}\delta A^{ij}(T,t)dt\biggr] \\ 
\displaystyle\text{ \ \ }+E_{\tilde{P}}\biggl[\biggl(\mu \left( 0\right) +%
\underset{j=1}{\overset{d}{\sum }}\int_{0}^{T}\eta ^{j}(t)dB_{t}^{j}+%
\underset{j=1}{\overset{d}{\sum }}\int_{0}^{T}dM_{s}^{j}\biggr)\underset{i=1}%
{\overset{d}{\sum }}\int_{0}^{T}\delta \sigma ^{i}(T,t)dB_{t}^{i}\biggr] \\ 
\displaystyle=E_{\tilde{P}}\biggl[\mathbb{\varphi }_{x}(\bar{X}(T))\biggl(%
\underset{i,j=1}{\overset{d}{\sum }}\int_{0}^{T}\delta A^{ij}(T,t)dt+%
\underset{i=1}{\overset{d}{\sum }}\int_{0}^{T}\delta \sigma
^{i}(T,t)dB_{t}^{i}\biggr)\biggr] \\ 
\displaystyle=E_{\tilde{P}}\left[ \mathbb{\varphi }_{x}(\bar{X}(T))\left(
X^{u}\left( T\right) -\bar{X}\left( T\right) \right) \right] .%
\end{array}
\label{SF2}
\end{equation}%
By analogous arguments, it holds that%
\begin{equation}
\begin{array}{l}
\displaystyle\underset{i,j=1}{\overset{d}{\sum }}E_{\tilde{P}}\biggl[%
\int_{0}^{T}\int_{t}^{T}\Bigl[\delta A^{ij}(s,t)^{T}Y(s)+\gamma ^{ij}\left(
t\right) \delta \sigma ^{i}(s,t)^{T}Z^{j}\left( s,t\right) \Bigr]dsdt\biggr]
\\ 
\displaystyle=\underset{i,j=1}{\overset{d}{\sum }}E_{\tilde{P}}\biggl[%
\int_{0}^{T}\int_{s}^{T}\Bigl[\delta A^{ij}(t,s)^{T}Y(t)+\gamma ^{ij}\left(
s\right) \delta \sigma ^{i}(t,s)^{T}Z^{j}\left( t,s\right) \Bigr]dtds\biggr]
\\ 
\displaystyle=\underset{i,j=1}{\overset{d}{\sum }}E_{\tilde{P}}\biggl[%
\int_{0}^{T}\int_{0}^{t}\delta A^{ij}(t,s)^{T}Y(t)dsdt\biggr]+\underset{i,j=1%
}{\overset{d}{\sum }}E_{\tilde{P}}\biggl[\int_{0}^{T}\int_{0}^{t}\delta
\sigma ^{i}(t,s)^{T}Z^{j}\left( t,s\right) d\left\langle
B^{i},B^{j}\right\rangle _{s}dt\biggr] \\ 
\displaystyle=\underset{i,j=1}{\overset{d}{\sum }}E_{\tilde{P}}\biggl[%
\int_{0}^{T}\biggl\langle Y(t),\int_{0}^{t}\delta A^{ij}(t,s)ds\biggr\rangle %
dt\biggr]+\underset{i,j=1}{\overset{d}{\sum }}E_{\tilde{P}}\biggl[%
\int_{0}^{T}\biggl\langle\int_{0}^{t}Z^{j}\left( t,s\right)
dB_{s}^{j},\int_{0}^{t}\delta \sigma ^{i}(t,s)dB_{s}^{i}\biggr\rangle dt%
\biggr] \\ 
\displaystyle=E_{\tilde{P}}\biggl[\int_{0}^{T}\biggl\langle Y(t),\underset{%
i,j=1}{\overset{d}{\sum }}\int_{0}^{t}\delta A^{ij}(t,s)ds\biggr\rangle dt%
\biggr] \\ 
\displaystyle\text{ \ \ }+E_{\tilde{P}}\biggl[\int_{0}^{T}\biggl\langle E_{%
\tilde{P}}\left[ Y(t)\right] +\underset{j=1}{\overset{d}{\sum }}%
\int_{0}^{t}Z^{j}\left( t,s\right) dB_{s}^{j}+\underset{j=1}{\overset{d}{%
\sum }}\int_{0}^{t}dN^{j}\left( t,s\right) ,\underset{i=1}{\overset{d}{\sum }%
}\int_{0}^{t}\delta \sigma ^{i}(t,s)dB_{s}^{i}\biggr\rangle dt\biggr] \\ 
\displaystyle=E_{\tilde{P}}\biggl[\int_{0}^{T}\biggl\langle Y(t),\underset{%
i,j=1}{\overset{d}{\sum }}\int_{0}^{t}\delta A^{ij}(t,s)ds+\underset{i=1}{%
\overset{d}{\sum }}\int_{0}^{t}\delta \sigma ^{i}(t,s)dB_{s}^{i}%
\biggr\rangle dt\biggr] \\ 
\displaystyle=E_{\tilde{P}}\biggl[\int_{0}^{T}\Bigl\langle Y(t),X^{u}\left(
t\right) -\bar{X}\left( t\right) \Bigr\rangle dt\biggr].%
\end{array}
\label{SF3}
\end{equation}%
Combining (\ref{SF1})-(\ref{SF3}), we have%
\begin{equation}
\begin{array}{l}
\displaystyle E_{\tilde{P}}\biggl[\int_{0}^{T}H^{u}(t)-\bar{H}(t)dt\biggr]%
=E_{\tilde{P}}\biggl[\int_{0}^{T}g\left( t,X^{u}\left( t\right) ,u\left(
t\right) \right) -g\bigl(t,\bar{X}\left( t\right) ,\bar{u}\left( t\right) %
\bigr)dt\biggr] \\ 
\displaystyle\text{ \ \ \ \ \ \ \ \ \ \ \ \ \ \ \ \ \ \ \ \ \ \ \ \ \ \ \ \
\ \ \ \ \ \ \ }+E_{\tilde{P}}\Bigl[\mathbb{\varphi }_{x}(\bar{X}(T))\bigl(%
X^{u}\left( T\right) -\bar{X}\left( T\right) \bigr)\Bigr]+E_{\tilde{P}}%
\biggl[\int_{0}^{T}\left\langle Y(t),X^{u}\left( t\right) -\bar{X}\left(
t\right) \right\rangle dt\biggr].%
\end{array}
\label{SF4}
\end{equation}%
Furthermore, note that 
\begin{equation}
\bar{H}_{x}(t)=E_{\tilde{P}}^{\mathcal{F}_{t}}\left[ Y\left( t\right) \right]
=Y\left( t\right) ,  \label{SF5}
\end{equation}%
Note that $\tilde{P}\in \mathcal{\tilde{P}}\left( \bar{X},\bar{u}\right) ,$
we have 
\begin{equation}
\begin{array}{l}
\displaystyle J\left( u(\cdot )\right) -J\left( \bar{u}(\cdot )\right) \\
\geq
E_{\tilde{P}}\biggl[\mathbb{\varphi }(X^{u}(T))+\int_{0}^{T}g(s,X(s),u(s))ds%
\biggr]\mathbb{-}E_{\tilde{P}}\biggl[\mathbb{\varphi }(\bar{X}%
(T))+\int_{0}^{T}g(s,\bar{X}(s),\bar{u}(s))ds\biggr] \\ 
=\displaystyle E_{\tilde{P}}\biggl[\mathbb{\varphi }(X^{u}(T))-\mathbb{%
\varphi }(\bar{X}(T))\biggr]\mathbb{-}E_{\tilde{P}}\biggl[%
\int_{0}^{T}g(s,X(s),u(s))ds-\int_{0}^{T}g(s,\bar{X}(s),\bar{u}(s))ds\biggr]%
\end{array}
\label{SF6}
\end{equation}%
Since $H$ is convex in $x$,$u$ and $\varphi $ is convex in $x$,$u,$ we
deduce from (\ref{SF4})-(\ref{SF6}) that for each $\bar{u}(\cdot )\in 
\mathcal{U}[0,T]$,%
\begin{equation}
\begin{array}{l}
\displaystyle J\left( u(\cdot )\right) -J\left( \bar{u}(\cdot )\right) \\ 
\displaystyle\geq E_{\tilde{P}}\Bigl[\mathbb{\varphi }_{x}(\bar{X}(T))\left(
X^{u}\left( T\right) -\bar{X}\left( T\right) \right) \Bigr]+E_{\tilde{P}}%
\biggl[\int_{0}^{T}g\left( t,X^{u}\left( t\right) ,u\left( t\right) \right)
-g\left( t,\bar{X}\left( t\right) ,\bar{u}\left( t\right) \right) dt\biggr]
\\ 
\displaystyle=E_{\tilde{P}}\biggl[\int_{0}^{T}H^{u}(t)-\bar{H}(t)dt\biggr]%
-E_{\tilde{P}}\biggl[\int_{0}^{T}\Bigl\langle Y(t),\left( X^{u}\left(
t\right) -\bar{X}\left( t\right) \right) \Bigr\rangle dt\biggr] \\ 
\displaystyle\geq E_{\tilde{P}}\biggl[\int_{0}^{T}\bar{H}_{x}(t)\left(
X^{u}\left( t\right) -\bar{X}\left( t\right) \right) dt\biggr]-E_{\tilde{P}}%
\biggl[\int_{0}^{T}\Bigl\langle Y(t),\left( X^{u}\left( t\right) -\bar{X}%
\left( t\right) \right) \Bigr\rangle dt\biggr] \\ 
\displaystyle\text{ \ \ }+E_{\tilde{P}}\Biggl[\int_{0}^{T}\bar{H}%
_{u}(t)\left( u\left( t\right) -\bar{u}\left( t\right) \right) dt\biggr] \\ 
\displaystyle=E_{\tilde{P}}\biggl[\int_{0}^{T}\bar{H}_{u}(t)\left( u\left(
t\right) -\bar{u}\left( t\right) \right) dt\biggr]\geq 0,%
\end{array}%
\end{equation}%
which implies that $\bar{u}(\cdot )$ is the optimal control. This completes
the proof.
\end{proof}

\section*{Declarations}

\noindent {\large \textbf{Competing interests}} The authors have no
competing interests to declare that are relevant to the content of this
article.

\noindent {\large \textbf{Funding}} The work was supported by the National
Natural Science Foundation of China (No. 12326603, 11671231).

\noindent {\large \textbf{Data availability statement}} There is no
associated data in this study.

\noindent {\large \textbf{Author Contributions}} All authors reviewed the
manuscript.


\begin{thebibliography}{99}
\bibitem{berger1980volterra}M. Berger, V. Mizel, Volterra equations with It\^{o} integrals, I, II, J. Integral Equations. 2 (1980) 187–245, 319–337.

\bibitem{Bagiani2014_Gcontrol}F. Bagiani, T. Meeyeer-Brandis, B. Oksendal, Optimal control with delayed information  flow of systems driven by $G$-Brownian motion, PUQR, 3(1) (2014).

\bibitem{Buckdan2025meanG} R. Buckdahn, B. He, J. Li, Mean field stochastic
control under sublinear expectation, SIAM J. Control Optim. 63 (2025)
1051-1084.

\bibitem{Denis2011function} L. Denis, M. Hu, S. Peng, Function spaces and
capacity related to a sublinear expectation: application to $G$-Brownian
motion pathes, Potential Anal. 34 (2011) 139-161.

\bibitem{Fan_wang2025}S. Fan, T. Wang, J. Yong, Multi-dimensional super-linear backward stochastic Volterra integral equations, J. Diﬀer. Equ. 437 (2025) 1–63.

\bibitem{gao2009pathwise}F. Gao, Pathwise properties and homomorphic flows for stochastic differential equations driven by $ G $-Brownian motion, Stochastic Process. Appl. 119 (10) (2009) 3356–3382. 

\bibitem{HuY2019linear} Y. Hu, B. $\emptyset $ksendal, Linear Volterra
backward stochastic integral equations, Stoch. Process. Appl. 129 (2019)
626-633.

\bibitem{Hu2014a} M. Hu, S. Ji, S. Peng, Y. Song, Backward stochastic
differential equations driven by $G$-Brownian motion, Stoch. Proc. Appl. 124
(2014) 759-784.

\bibitem{Hu2016maximum} M. Hu, S. Ji, Stochastic maximum principle for
stochastic recursive optimal control problem under volatility ambiguity,
SIAM J. of Control Optim. 54 (2) (2016) 918-945.

\bibitem{Hu2016Qc} M. Hu, F. Wang, G. Zheng, Quasi-continuous random
variables and processes under the $G$-expectation framework, Stoch. Proc.
Appl. 126 (2016) 2367-2387.

\bibitem{Qutime_vary}M. Hu, B. Qu, F.Wang, BSDEs driven by $G$-Brownian motion with time-varying Lipschitz condition, J. Math. Anal. Appl. 491 (2020) 1–27. 

\bibitem{Hamaguchi_delay}Y. Hamaguchi, On the maximum principle for optimal control problems of stochastic Volterra integral equations with delay, Appl. Math. Optim. 87 (42) (2023).

\bibitem{ito1979on}I. Ito, On the existence and uniqueness of solutions of stochastic integral equations of the Volterra type, Kodai Math. J. 2 (1979) 158-170.

\bibitem{Jiangl2023stable} L. Jiang, G. Liang, A robust $\alpha $-stable
central limit theorem under sublinear expectation without integrability
condition, J. Theoretic. Probab. 37 (2024) 2394-2424.

\bibitem{shu_K}I. Karatzas, S. Shreve, Brownian Motion and Stochastic Calculus, Springe-Verlag, New  York, 1991.

\bibitem{Karoui1997bsde_N}N. El Karoui, S. Huang, A general result of existence and uniqueness of backward stochastic differential equations, in Backward Stochastic Differential Equations, N. El Karoui, and L. Mazliak, eds., Pitman Res. Notes Math. Ser., 364 (1997) 27–36.

\bibitem{Lin2002adapted} J. Lin, Adapted solution of a backward stochastic
nonlinear Volterra integral equation, Stoch. Anal. Appl. 20 (2002) 165-183.

\bibitem{LiuG2020multi} G. Liu, Multi-dimensional BSDEs driven by $G$%
-Brownian motion and related system of fully nonlinear PDEs, Stochastics 92
(5) (2020) 659-683.

\bibitem{Linyq2019reflect} Y. Lin, A.S. Hima, Reflected stochastic
differential equations driven by $G$-Brownian motion in non-convex domains,
Stoch. Dyn. 19 (03) (2019) 1950025.

\bibitem{luo2014stochastic}P. Luo, F. Wang, Stochstic differential equations driven by $G$-Brownian motion and ordinary differential equations. Stochastic Process. Appl. 124 (11) (2014) 3869–3885.

\bibitem{pardoux1990stochastic}E. Pardoux, P. Protter, Stochastic Volterra equations with anticipating coefficients. Ann. Probab. 18 (1990) 1635–1655.

\bibitem{peng2019nonlinear} S. Peng, Nonlinear Expectations and Stochastic
Calculus under Uncertainty, Springer, Berlin (2019).

\bibitem{popier2021_N}A. Popier, Backward stochastic Volterra integral equations with jumps in a general filtration. ESAIM Probab. Stat. 25 (2021) 133–203.

\bibitem{Ren2010jump} Y. Ren, On solutions of backward stochastic Volterra
integral equations with jumps in Hilbert space, J. Optim. Theory Appl. 144
(2010) 319-333.

\bibitem{Soner2012Bsde2} H. M. Soner, N. Touzi, J. Zhang, Wellposedness of
Second Order Backward SDEs, J. Theoret. Probab. 153 (1-2) (2012) 149-190.

\bibitem{WangT2013mean} Y. Shi, T. Wang, J. Yong, Mean-field backward
stochastic Volterra integral equations, Discrete Contin. Dyn. Syst. Ser. B 18
(7) (2013) 1929-1967.

\bibitem{Wen_shi_double} Y. Shi, J. Wen, J. Xiong, Backward doubly stochastic Volterra integral equations and their applications, J. Diﬀer. Equ. 269 (2020) 6492–6528.

\bibitem{Sun_G_control}Z. Sun. Maximum principle for forward-backward stochastic control system under $G$-Brownian motion and relation to dynamic programming, J. of Comput. and Appl. Mathematics 296 (2016) 753-775.

\bibitem{WangH2021risk} H. Wang, J. Sun, J. Yong, Recursive utility
processes, dynamic risk measures and quadratic backward stochastic Volterra
integral equations, Appl. math. Optim. 84 (2021) 145-190.

\bibitem{WangT2017forward-backward} T. Wang, H. Zhang, Optimal control
problems of forward-backward stochastic Volterra integral equations with
closed control regions, SIAM J. Control Optim. 55 (2017) 2574-2602.

\bibitem{Wangzhang2007} Z. Wang, X. Zhang, Non-Lipschitz backward stochastic
Volterra type equations with jumps, Stoch. Dyn. 7 (2007) 479-496.

\bibitem{xu2014G_control}Y. Xu, Stochastic maximum principle for optimal control with multiple priors, Systems Control  Lett. 64 (2014) 114–118.

\bibitem{yong2008Wellposedness} J. Yong, Well-posedness and regularity of
backward stochastic Volterra integral equations, Probab. Theory Relat.
Fields 142 (2008) 21-77.

\bibitem{yong2006Wellposedness} J. Yong, Backward stochastic Volterra
integral equations and some related problems, Stoch. Process. Appl. 116
(2006) 779-795.

\bibitem{zhao} B. Zhao, R. Li, M. Hu, Stochastic Volterra Integral Equations
Driven by $G$-Brownian Motion, Math. Methods Appl. Sci. (2025) 1-19.
\end{thebibliography}
\end{document}